\documentclass[letterpaper, reqno, final]{amsart}

\pdfoutput=1 %For arXiv

\usepackage{amsmath, amssymb, amsthm}
\usepackage{dsfont}
\usepackage[hidelinks]{hyperref}

\usepackage{graphicx}
\usepackage{cite}
\usepackage{enumerate}
\usepackage{bookmark}
\usepackage{tikz-cd} 
\usepackage{subcaption}
\usepackage{mathtools}
\usepackage{cleveref}

\crefformat{equation}{(#2#1#3)}
\crefrangeformat{equation}{(#3#1#4)--(#5#2#6)}
\crefmultiformat{equation}{(#2#1#3)}%
{ and~(#2#1#3)}{, (#2#1#3)}{ and~(#2#1#3)}

\crefrangeformat{thm}{Theorems #3#1#4 to #5#2#6}
\crefmultiformat{thm}{Theorems #2#1#3}%
{ and~#2#1#3}{, #2#1#3}{ and~#2#1#3}

\calclayout

\usepackage[activate={true,nocompatibility},final,tracking=true,kerning=true,spacing=true,factor=1100,stretch=20,shrink=20]{microtype}
\microtypecontext{spacing=nonfrench}

\usepackage[right]{showlabels}

\DeclareMathOperator{\supp}{supp}

\DeclareMathOperator{\sgn}{sgn}

\DeclareMathOperator*{\slim}{s-lim}
\DeclareMathOperator*{\wlim}{w-lim}

\newcommand{\jBra}[1]{\langle #1 \rangle}

\newcommand{\editadd}[1]{{ #1}}

\newcommand{\bbR}{\mathbb{R}}

\newcommand{\bbC}{\mathbb{C}}
\newcommand{\bbOne}{\mathds{1}}

\newcommand{\fu}{\mathfrak{u}}
\newcommand{\fv}{\mathfrak{v}}

\newcommand{\urem}{u_\textup{rem}}
\newcommand{\uloc}{u_\textup{loc}}
\newcommand{\Jfree}{J_\textup{free}}

\newcommand{\rmI}{\mathrm{I}}

\newtheorem{thm}{Theorem}

\newtheorem{prop}[thm]{Proposition}
\newtheorem{lemma}[thm]{Lemma}
\newtheorem{cor}[thm]{Corollary}
\newtheorem{rmk}{Remark}

\begin{document}
\title[Scattering and localized states for defocusing NLS with potential]{Scattering and localized states for defocusing nonlinear Schr\"odinger equations with potential
}
\author[Avy Soffer]{Avy Soffer$^{*}$}
 \address[Avy Soffer]{\newline
        Department of Mathematics, \newline
         Rutgers University, New Brunswick, NJ 08903 USA.}
  \email[]{soffer@math.rutgers.edu}
\author[Gavin Stewart]{Gavin Stewart}
 \address[Gavin Stewart]{\newline
        Department of Mathematics, \newline
         Rutgers University, New Brunswick, NJ 08903 USA.}
  \email[]{gavin.stewart@rutgers.edu}
\thanks{
A.Soffer is supported in part by NSF-DMS Grant number 2205931
}
\thanks{$^{*}$: Corresponding author.  Email: \href{mailto:soffer@math.rutgers.edu}{\url{soffer@math.rutgers.edu}}}
\date{\today}
%\maketitle
\begin{abstract}
 We study the large-time behavior of global energy class ($H^1$) solutions of the one-dimensional nonlinear Schr\"odinger equation with a general localized potential term and a defocusing nonlinear term. By using a new type of interaction Morawetz estimate localized to an exterior region, we prove that these solutions decompose into a free wave and a weakly localized part which is asymptotically orthogonal to any fixed free wave.  We further show that the $L^2$ norm of this weakly localized part is concentrated in the region $|x| \leq t^{1/2+}$, and that the energy ($\dot{H}^1$) norm is concentrated in $|x| \leq t^{1/3+}$.  Our results hold for solutions with arbitrarily large initial data.
\end{abstract}

\maketitle

\section{Introduction}
In this work, we study the large-time behavior of solutions to the following nonlinear Schr\"odinger equation (NLS) in dimension 1:
\begin{equation}\label{eqn:NLS}
    i \partial_t u  =- \Delta u + V(x,t)u + |u|^{p-1}u
\end{equation}
The interaction consists of two terms: a potential term localized near $x = 0$ and a defocusing nonlinearity.  This type of equation occurs in many physical models from areas such as optics \cite{barak2008observation} and condensed matter physics~\cite{kengne2021spatiotemporal,wang2021rogue}. 
  We will focus on solutions that are globally bounded in the $H^1$ norm, with no additional smallness assumptions.%The data can be arbitrary large in $H^1.$

A fundamental goal of mathematical physics is to \editadd{describe all possible asymptotic behaviors for a given equation}.  This problem, known as asymptotic completeness, has been studied extensively in the case of linear equations with time-independent interactions, and very general results have been proved in both the two-body and $N$-body cases: see the texts~\cite{derezinskiScatteringTheoryClassical1997,yafaevMathematicalScatteringTheory2010,reedMethodsModernMathematical2003} and references therein.  For the time-independent problem, spectral theory tells us that the solution separates into some number of bound states (given by eigenfunctions in the discrete spectrum of $-\Delta + V$) and radiation (given by projection to the continuous spectrum).  The time evolution of the bound states is explicit, and for potentials with sufficient decay at infinity the radiation scatters to a free wave that solves the equation with potential $V = 0$.

However, once the interaction terms are time-dependent or nonlinear, there is no asymptotic theory in anything approaching the generality of the linear case.  Leaving aside the case of completely integrable equations, we cannot expect the evolution of the bound states and radiation to decouple like in the linear case.  It is not even clear whether the bound states and radiation decouple asymptotically, since the solution could have components that spread slowly in time~\cite{soffer2022existence}.  Despite these obstacles, a growing body of numerical and theoretic work has led to the widespread belief that (generic) solutions for a wide array of nonlinear dispersive equations resolve asymptotically into a number of bound states (called solitons) and radiation~\cite{sofferSolitonDynamicsScattering2006}.  Although this soliton resolution conjecture has inspired a great deal of work on the existence and stability of solitons, it has only been proven to hold for a handful of equations~\cite{duyckaertsSolitonResolutionRadial2019,deiftSteepestDescentMethod1993}.  

%{Previous works on the asymptotics of~\eqref{eqn:NLS} have focused on either proving scattering for repulsive potentials and nonlinearities, or }.  If the potential is assumed to be time-independent and repulsive, it was proved in~\cite{danconaScatteringDefocusingNLS2023} {that solutions scatter}.  There have also been recent results concerning the asymptotic stability of solitons under small, localized perturbations~\cite{chenLongtimeDynamicsSmall2023,collotAsymptoticStabilitySolitary2023}.  In contrast to these works, our results do not require any repulsiveness for the potential (and in particular hold if $-\Delta + V$ has a bound state), and the only assumption made on the data is that the energy remains bounded for all time:

In fact, the standard scattering problem for linear localized time-dependent potentials has been largely open until very recently.  There are fundamental reasons for this. The separation between the radiation and bound state components of the solutions cannot be determined by spectral theory and may be arbitrarily slow.  A key problem is to prove microlocal propagation estimates to control the behavior of the solution in certain regions of phase-space.  Although these types of estimate are well understood in the time-independent case, where they were decisive in proving $N$-body scattering~\cite{sigal1987n}, the derivation of these estimates is significantly harder in the time-dependent case owing to the inapplicability of spectral theory methods, which prevent us from projecting away bound states and low-energy waves.  

{In our work, we will study the asymptotics of solutions $u(t)$ to~\eqref{eqn:NLS} that satisfy the condition}
\begin{equation}\label{eqn:bdd-energy-hypo}
    \sup_t \lVert u(t) \rVert_{H^1} < \infty
\end{equation}
{This assumption is standard in the field, going back to the work of Tao~\cite{taoAsymptoticBehaviorLarge2004b}.  Moreover, for the defocusing equation~\eqref{eqn:NLS}, it is equivalent to assuming that the energy}
\begin{equation*}
    {\mathcal{E}(t) = \int \frac{1}{2} |\nabla u|^2 + V(x,t) |u(x,t)|^2 + \frac{1}{p+1} |u(x,t)|^{p+1}\;dx}
\end{equation*}
is globally bounded, since by the triangle inequality
\begin{equation*}
    {\lVert u(t) \rVert_{H^1}^2 \leq 2 \mathcal{E}(t) + \lVert V(x,t) \rVert_{L^\infty_{t,x}}\lVert u(t) \rVert_{L^2}^2}
\end{equation*}
{This is a realistic assumption, and can be shown to hold for a number of physically relevant models, including the cases where $V$ is time-independent or satisfies the adiabatic condition $\partial_t V \in L^1_tL^\infty_x$.}

Under the assumption~\eqref{eqn:bdd-energy-hypo}, Tao has shown that in high dimensions ($d \geq 5$) a weak form of soliton resolution holds, in the sense that all solutions to the nonlinear Schr\"odinger equation decompose in $H^1$ as time goes to infinity as
\begin{equation}\label{eqn:wsr-decomp}
    u(t) = \sum_{i=1}^N \uloc^{(i)}(x - x_i(t), t) + e^{it\Delta} u_+ + o(1)
\end{equation}
where each function $\uloc^{(i)}$ satisfies the uniform-in-time localization condition
\begin{equation}\label{eqn:wsr-loc}
    \lim_{R \to \infty} \sup_t \int_{|x| \geq R} |\uloc^{(i)}(x,t)|^2\;dx = 0
\end{equation}
and $|x_i(t) - x_j(t)| \to \infty$ for $i \neq j$~\cite{taoConcentrationCompactAttractor2007}.  Under the additional assumption that $u$ is radial, the decomposition~\eqref{eqn:wsr-decomp} simplifies t
\begin{equation}\label{eqn:wsr-rad-decomp}
    u(t) = \uloc(x, t) + e^{it\Delta} u_+ + o_{H^1}(1)
\end{equation}
and $\uloc$ satisfies the quantitative localization bounds
\begin{equation}\label{eqn:wsr-mass-loc}
    \sup_t\int_{|x| \geq R} |\uloc(x,t)|^2\;ds \lesssim_{u_0} R^{-\alpha}
\end{equation}
for some $\alpha$ depending only on the dimension~\cite{taoGlobalCompactAttractor2008b,soffer2023soliton}.  In lower dimensions, it has not yet been possible to obtain such precise asymptotics.  In dimension $3$, Tao's work~\cite{taoAsymptoticBehaviorLarge2004b} shows that solutions to the radial focusing NLS decompose as
\begin{equation}\label{eqn:wsr-rad-3D-decomp}
    u(t) = \uloc(x, t) + e^{it\Delta} u_+ + o_{\dot{H}^1}(1)
\end{equation}
where $|\nabla^j \uloc(x,t)| \lesssim \jBra{x}^{-3/2-j+}$.  In particular, this implies that the kinetic energy of the bound state $\uloc$ decays in $x$:
\begin{equation*}
    \sup_t\int_{|x| \geq R} |\nabla \uloc(x,t)|^2\;ds \lesssim_{u_0} R^{-2+}
\end{equation*}
which can be seen as a weaker version of the quantitative localization of mass~\eqref{eqn:wsr-mass-loc} for $d \geq 5$.  Although the mass is not known to be localized, the results of~\cite{liuLargeTimeAsymptotics2025,soffer2022large} show that it can only spread slowly over time.  More precisely, all the mass of the nonscattering part of the solution must be localized in the region $|x| \leq t^{1/2+}$ as $t \to \infty,$ in the sense that
\begin{equation}\label{eqn:wsr-weak-mass-loc}
    \lim_{t \to \infty} \int_{|x| \geq t^{1/2+}} |\uloc(x,t)|^2\;dx = 0
\end{equation}
Within the class of equations
\begin{equation}\label{eqn:Schro-tdp}
    i\partial_t u + \Delta u = V(x,t) u
\end{equation}
in dimension \editadd{at least 5} with $\jBra{x}^2V(x,t) \in L^\infty_{x,t}$, explicit self-similar solutions were constructed in~\cite{soffer2022existence}.  \editadd{Although such explicit constructions are not known in lower dimensions, the results suggest} that the exponent $1/2$ in~\eqref{eqn:wsr-weak-mass-loc} may be optimal.

Extending these results to dimensions one and two is difficult because the linear dispersive decay ${\lVert e^{it\Delta} \rVert_{L^1 \to L^\infty} \lesssim t^{-d/2}}$ is no longer integrable in time, which causes serious difficulties in even defining the linear scattering state $u_+$.  This problem was only recently addressed by~\cite{soffer2022large} for the equation~\eqref{eqn:Schro-tdp} by making use of the fact that
$$\lVert F(|x| \leq t^\alpha) F(|D| \geq t^{-\delta}) e^{-it\Delta} V(x,t) \rVert_{L^2_x \to L^2_x}$$
is integrable in time for potentials $V$ which decay pointwise at an appropriate rate.  {Here, $F(|x| \leq t^\alpha)$ and $F(|D| \geq t^{-\delta})$ denote smooth projections to the regions $|x| \leq t^\alpha$ and $|D| \geq t^{-\delta}$; see~\Cref{sec:def} for a precise definition.}  Using this, it is possible to prove that~\eqref{eqn:wsr-weak-mass-loc} holds for~\eqref{eqn:Schro-tdp} in dimensions one and two.

This method cannot be applied to~\eqref{eqn:NLS} due to the presence of the nonlinear term $|u|^{p-1}u$, which does not decay pointwise in dimension one.  To handle this term, we will use the defocusing character of the nonlinearity to prove Morawetz and interaction Morawetz estimates in the exterior region $|x| \geq t^{1/3+}$.  We then apply these estimates to two cases of~\eqref{eqn:NLS}: equations with nonradiative solutions and mass-supercritical equations.  If we assume that $u(t)$ is nonradiative (which, roughly speaking, means $u_+ = 0 :$ see~\Cref{thm:nonrad-spreading} for a precise definition), then the exterior Morawetz estimate can be used to show that
\begin{equation}\label{eqn:nonrad-spreading-intro}
    \int |x||u(x,t)|^2\;dx \lesssim \jBra{t}^{1/2},
\end{equation}
which in particular implies~\eqref{eqn:wsr-weak-mass-loc}.  For equations with $p < \editadd{5}$, nonlinear effects are strong enough to improve the right-hand side of~\eqref{eqn:nonrad-spreading-intro} to $\jBra{t}^{\editadd{\frac{2(p-1)}{3p+1}}}$.  For mass-supercritical equations, the exterior interaction Morawetz estimate gives us time-averaged decay for the nonlinearity, allowing us to prove the existence of the free channel wave operator $u_0 \to u_+$, as well as the improved localization for the weakly localized part $\uloc(t)$.  To the best of our knowledge, the exterior interaction Morawetz estimate we give is new, and we believe it may be of independent interest.

\subsection{Related work}

In addition to the work discussed above, there are a number of other perspectives on scattering and soliton resolution.  We will briefly summarize a few areas and discuss how they relate to our work below:

\subsubsection{Equations with repulsive interactions} In the case where the attractive part of the interaction is absent or too weak to form a bound state, profile decomposition can be used to prove scattering of the solution.  The general form of the argument, often known as the concentration compactness-rigidity method, was introduced in~\cite{kenigGlobalWellposednessScattering2006a} for the focusing energy-critical nonlinear Schr\"odinger equation in dimension $3$, and has since been generalized and refined in a vast number of works: see~\cite{kenigGlobalWellposednessScattering2008,duyckaertsScatteringNonradial3D2008, duyckaertsProfileDecompositionScattering2024a,fangScatteringFocusingEnergysubcritical2011,duyckaertsMassenergyScatteringCriterion2024} for a (very) incomplete sampling.  In particular, it can be applied in the presence of a small or repulsive potential~\cite{hongScatteringNonlinearSchrodinger2016, lafontaineScatteringNLSPotential2016, banicaScatteringNLSDelta2016,danconaScatteringDefocusingNLS2023}.  However, it is not applicable to~\eqref{eqn:NLS} when $V$ is allowed to be time-dependent or nonrepulsive, since the solution may exhibit bound states or other nonscattering behavior.

\subsubsection{NLS with potential and asymptotic stability of solitons}  Beginning with the pioneering work of~\cite{germainNonlinearResonancesPotential2015}, there have been a number of works on the asymptotic behavior of nonlinear Schr\"odinger equations with potential~\cite{germainNonlinearSchrodingerEquation2018,naumkinSharpAsymptoticBehavior2016,naumkinNonlinearSchrodingerEquations2018,chen1dimensionalNonlinearSchrodinger2022,chen1dCubicNLS2024,stewartAsymptoticsCubic1D2024} culminating in proofs that solitons for a wide range of 1 dimension nonlinear Schr\"odinger equations are asymptotically stable under small, localized perturbations~\cite{chenLongtimeDynamicsSmall2023,collotAsymptoticStabilitySolitary2023,liAsymptoticStabilitySolitary2024a}.  The argument proceeds by writing the solution $u$ as the sum of a modulated soliton and a perturbation.  In a reference frame moving with the soliton, the perturbation solves a nonlinear Schr\"odinger equation with potential.  By adapting the method of space-time resonances introduced in~\cite{germainGlobalSolutions3D2008}, it is possible to show that small perturbations exhibit modified scattering.  A more complete summary of the method can be found in the recent survey article~\cite{germainReviewAsymptoticStability2024}.  This method gives very precise asymptotics; however, it necessarily requires the initial data to be a small, localized perturbation of a soliton and thus cannot be applied to nonlinear equations with generic large initial data.  Moreover, the method relies crucially on the linear spectral theory of Schr\"odinger operators with potential, so it is not clear how to adapt the argument to accommodate equations like~\eqref{eqn:NLS} where the potential is time dependent.  Even in the case of time-independent potentials, one-dimensional models containing internal modes have been stubbornly resistant to this style of analysis~\cite{delortLongtimeDispersiveEstimates2022,legerInternalModesRadiation2021}.

\subsubsection{Virial methods and local asymptotic stability of solitons} Beginning with the work of~\cite{kowalczykKinkDynamicsModel2017a}, there have been a number of recent works that have established local asymptotic stability for a number of models containing solitons~\cite{martelAsymptoticStabilitySmall2024,martelAsymptoticStabilitySolitary2023,riallandAsymptoticStabilitySolitary2024,riallandAsymptoticStabilitySolitons2024,cuccagnaAsymptoticStabilityGround2024,cuccagnaAsymptoticStabilityLine2024}.  The key ingredient in this argument is to prove an appropriate virial identity, which gives decay of solutions in an appropriate (interior) region of space-time.  This method has proved quite flexible, especially in applications to models that exhibit strongly nonlinear behavior.  However, the information it gives is inherently local in character.

\subsection{Main results}

We will assume that the potential and its derivative obey {uniform in time spatial} decay bounds
\begin{equation}\label{eqn:V-decay-hypo}
   \sup_{{x},t} |\jBra{x}^\sigma V(x,t)| \leq C
\end{equation}
and
\begin{equation}\label{eqn:dV-decay-hypo}
    \sup_{{x},t}|\jBra{x}^{\sigma+1} \partial_x V(x,t)| \leq C
\end{equation}

Our first result concerns the behavior of solutions that behave like a purely bound state, with no freely scattering radiation as $t\to \infty$.  \editadd{To state the result, we first recall some terminology related to propagation estimates.  For an operator $B$, we define the expectation of $B$ with respect to $u$ at time $t$ to be
\begin{equation}\label{eqn:prop-obs-def}
    \langle B \rangle_{t,u} = \langle Bu, u \rangle_{L^2}
\end{equation}
when the function $u$ is clear from context, we will often drop it and simply write $\langle B \rangle_t$.  The time evolution of a (possibly time-dependent) operator $B(t)$ for a solution to~\eqref{eqn:NLS} is given by
\begin{equation}\label{eqn:obs-deriv}
    \partial_t \langle B(t) \rangle_t = \langle D_H B(t) \rangle_t
\end{equation}
where $D_H B(t)$ is the Heisenberg derivative
\begin{equation}\label{eqn:Heis-deriv}
    D_H B(t) = [-i\Delta + iV + i|u|^{p-1}, B(t)] + \partial_t B(t)
\end{equation}
By judiciously choosing $B$ so that the commutator term in~\eqref{eqn:Heis-deriv} is positive (up to integrable perturbations), it is often possible to obtain global decay estimates, which we will call propagation estimates.  The celebrated Morawetz estimate for nonlinear Schr\"odinger equations is an important type of propagation estimate.  Observe that for $B(t) = |x|$, we have
\begin{equation}\label{eqn:gamma-def}
    D_H |x| = [-i\Delta, |x|] = 2\left(\frac{\sgn(x) \partial_x + \partial_x \sgn(x)}{2i}\right) =: 2 \gamma
\end{equation}
where $\langle \gamma \rangle_t$ is the Morawetz action.  Applying~\eqref{eqn:obs-deriv} and~\eqref{eqn:Heis-deriv} to $\langle \gamma \rangle_t$ (with $V$ repulsive) gives the standard Morawetz estimates.  The operator $\gamma$ can be understood to measure the degree to which solutions travel radially away from the origin.

For a solution $u_\text{free} = e^{it\Delta} u_0$ of the free Schr\"odinger equation, we have that 
\begin{equation}\label{eqn:gamma-free-lim}
    \lim_{t \to \infty} \langle \gamma \rangle_{t, u_\text{free}} = \lVert u \rVert_{\dot{H}^{1/2}}^2
\end{equation}   
Indeed, we have that
 $e^{-it\Delta} \frac{x}{t} e^{it\Delta} \to 2 D$ in the $H^1$ strong resolvent sense (see~\cite[Lemma 4.17]{liuLargeTimeAsymptotics2025}).  In particular, this implies that $f(e^{-it\Delta} \frac{x}{t} e^{it\Delta}) \to f(2D)$ for any continuous function $f$ (cf.~\cite[Theorem VIII.20]{reedMethodsModernMathematical1980}).  Thus, for any $\epsilon > 0$,
\begin{equation*}
    \slim_{t \to \infty} e^{-it\Delta} F\left(|x| \geq 2\epsilon t\right) \gamma F\left(|x| \geq 2\epsilon t\right) e^{it\Delta} =  F(|D| \geq \epsilon) |D| F(|D| \geq \epsilon) 
\end{equation*}
On the other hand, the strong minimum velocity bounds imply that
\begin{equation*}
    \lim_{\epsilon \to 0} \limsup_{t \to \infty} \lVert F\left(|x| \leq 2 \epsilon t\right) u_\text{free} \rVert_{H^1} = 0
\end{equation*}
so, taking $\epsilon \to 0$, we obtain~\eqref{eqn:gamma-free-lim}.  By adapting the argument, we also see that for any $\beta \in (0,1)$,
$$\lim_{t \to \infty} \langle F(|x| \geq t^\beta) \gamma F(|x| \geq t^\beta) \rangle_{t, u_\text{\editadd{free}}} = \lVert u_0 \rVert_{\dot{H}^{1/2}}^2$$
On the other hand, if $u_\text{bound} = e^{-it(-\Delta + V)} u_0$ for some $u_0$ in the discrete spectral subspace of $-\Delta + V$, then the RAGE theorem immediately implies that
$$\lim_{t \to \infty} \langle F(|x| \geq t^\beta) \gamma F(|x| \geq t^\beta) \rangle_{t, u_\text{\editadd{bound}}} = 0$$

By extending these arguments, we see that for a sufficiently nice time-independent potential $V$, if $u = e^{-it(-\Delta + V)} u_0$, then
\begin{equation}\label{eqn:nonrad-hypo}
    \lim_{t \to \infty} \langle F(|x| \geq t^\beta) \gamma F(|x| \geq t^\beta) \rangle_{t,u} = 0
\end{equation}
iff $u_0$ is in the discrete spectral subspace of $-\Delta + V$.  Such a precise characterization appear out of reach for~\eqref{eqn:NLS}; however, we are able to prove that any solution to~\eqref{eqn:NLS} which satisfies~\eqref{eqn:nonrad-hypo} must spread out in space more slowly than a free solution.
}

\begin{thm}\label{thm:nonrad-spreading}
    Suppose $\sigma \geq 2$, that $u$ is a bounded energy solution to~\eqref{eqn:NLS} with $p > 1$ and is \textit{nonradiative} in the sense that \editadd{there exists $\beta \in (1/3,1)$ such that~\eqref{eqn:nonrad-hypo} holds.}  If $|x|^{1/2}u_0 \in L^2$, then $\langle |x| \rangle_t \lesssim 1 + t^{1/2}$.  

    Furthermore, if $p < \editadd{5}$, then we have the improved bound $\langle |x| \rangle_t \lesssim 1 + t^\editadd{\frac{2(p-1)}{3p+1}}$\editadd{ for $p \in (7/3, 5)$ and $\langle |x| \rangle_t \lesssim 1 + t^{1/3+}$ for $1 < p < 7/3$.}
\end{thm}

\begin{rmk}
    Although the statement of the theorem only requires~\eqref{eqn:nonrad-hypo} to hold for an arbitrary $\beta \in (1/3, 1)$, we will see later in~\Cref{lem:all-beta-lem} that if~\eqref{eqn:nonrad-hypo} holds for a single $\beta \in (1/3, 1)$, then it holds for all $\beta$ in that range (see also~\cite[Section 4.3.2]{liuLargeTimeAsymptotics2025}).  In particular, we will be able to optimize over $\beta$ in our proof of~\Cref{thm:nonrad-spreading}.
\end{rmk}

\begin{rmk}
    It might seem paradoxical that decreasing $p$, which \textit{increases} the ability of the nonlinearity to disperse concentrated states, also \textit{decreases} in the speed at which nonradiative states spread.  This can be understood by noting that outgoing waves that are located far from the origin and have nonnegligible energy will be traveling fast enough to scatter freely.  Since the defocusing nonlinearity aids in this scattering, any weakly bound part of the solution must stay closer to the origin for small $p$ to avoid becoming radiation.
\end{rmk}

Next, we consider the problem~\eqref{eqn:NLS} with mass supercritical nonlinearity $p > 5$, and prove that the solutions split asymptotically into a free wave and a localized part which spreads slowly in time:

\begin{thm}\label{thm:wave-op-and-bdd-state}
    Suppose $\sigma \geq 2$ and $p > 5$  Then, any solution to~\eqref{eqn:NLS} satisfying~\eqref{eqn:bdd-energy-hypo} can be written as
    \begin{equation}\label{eqn:thm2-decomp}
        u(t) = e^{it\Delta} u_+ + \uloc(t)
    \end{equation}
    where $u_+ \in H^1$, and for any $\kappa > 1/2,$ and $\mu > \max(\frac{1}{3}, \frac{1}{\sigma})$, the weakly localized part $\uloc$ satisfies
    \begin{equation}\label{eqn:thm2-mass}
        \lim_{t \to \infty} \lVert F(|x| \geq t^{\kappa}) \uloc(t) \rVert_{L^2} = 0
    \end{equation}
    and
    \begin{equation}\label{eqn:thm2-energy}
       \lim_{t \to \infty} \lVert F(|x| \geq t^{\mu}) \partial_x \uloc(t) \rVert_{L^2} = 0
    \end{equation}
\end{thm}

{We note that for a free solution $u_\text{free} = e^{it\Delta} u_0$, we have that
$$\lim_{t \to \infty} \lVert F(|x| \leq t^{\kappa}) u_\text{free}(t) \rVert_{L^2} = \lim_{t \to \infty} \lVert F(|x| \leq t^{\mu}) \partial_x u_\text{free}(t) \rVert_{L^2} = 0$$
(this follows from~\Cref{lem:st-ph-lemma-high-freq} in~\Cref{sec:disp-est}).  In particular, this implies that the decomposition~\eqref{eqn:thm2-decomp} reflects an increasing separation between the free part $e^{it\Delta} u_+$ and the weakly localized part $\uloc(t)$ in physical space as $t \to \infty$.}

\editadd{We now briefly discuss the optimality of the parameters $\kappa$ and $\mu$.  At least for $d \geq 5$, it is known that for $\kappa < 1/2$, the model equation~\eqref{eqn:Schro-tdp} admits solutions which decompose asymptotically into free radiation and a localized part as in~\eqref{eqn:thm2-decomp}, but such that $\lim_{t \to \infty} F(|x| \geq t^\kappa) \uloc(t) \rVert_{L^2} \neq 0$; see~\cite{soffer2022existence}.  In particular, this suggests that $\kappa = 1/2$ is the optimal endpoint.  Constructing such solutions in dimension $1$ (even in the absence of a nonlinearity) is likely quite difficult due to the weak dispersion; however, we expect that $\kappa = 1/2$ is still optimal in dimension $1$.

It is more difficult to assess the optimality of $\mu$, especially for a nonlinear equation like~\eqref{eqn:NLS}.  However, based on analogous results~\cite{taoAsymptoticBehaviorLarge2004b} for the cubic NLS, we expect that in fact the energy of the localized part $\uloc$ does not spread, in the sense that
\begin{equation*}
    \limsup_{t \to \infty} \lVert F(|x| \geq R) \uloc(t) \rVert_{L^2} = o(1)
\end{equation*}
as $R \to \infty$.}

\begin{rmk} %TODO: Is it worth giving this more general case?
    In the interest of exposition, we have stated~\Cref{thm:nonrad-spreading} and~\Cref{thm:wave-op-and-bdd-state} with the short-range hypothesis $\sigma \geq 2$.  With a slight modification to the argument, we could also derive similar results (but with worse localization) under the optimal short-range hypothesis $\sigma > 1$.
\end{rmk}

\begin{rmk}
    The restriction $\mu > \frac{1}{3}$ arises because in the interaction Morawetz estimates we must commute the Laplacian through a spatial cutoff, producing an additional effective potential term of size $\jBra{x}^{-3}$.  This estimate is only necessary to control the defocusing nonlinearity, so our arguments also show that the rate at which energy spreads in the time-dependent linear problem is determined by the decay of the potential.
\end{rmk}

\begin{rmk}
    The proofs of~\Cref{thm:nonrad-spreading,thm:wave-op-and-bdd-state} ultimately rely on the exterior Morawetz estimates that we present in~\Cref{sec:PRES}.  Since Morawetz estimates hold for solutions in $H^{1/2}$, it is likely that~\eqref{eqn:bdd-energy-hypo} could be weakened to $\lVert u \rVert_{L^\infty_t H^{s}_x} < \infty$ for some $s > 1/2$ (and possibly $s = 1/2$, although one would have to be careful about the failure of the endpoint Sobolev embedding).  In the interest of expositional clarity, we do not pursue these improvements here.
\end{rmk}

\subsection{Difficulties and a sketch of the proof}
In one dimension, the dispersive decay of $e^{it\Delta}$ is too weak to apply the methods of~\cite{soffer2023soliton,liuLargeTimeAsymptotics2025,soffer2022large} to~\eqref{eqn:NLS} because the term $|u|^{p-1}u$ does not belong to any weighted $L^p$ space.  We will overcome this difficulty by deriving decay estimates in the exterior region $|x| \geq t^{\beta}$, which provide sufficient control far from the origin to prove~\Cref{thm:nonrad-spreading,thm:wave-op-and-bdd-state}.

We begin with~\Cref{thm:nonrad-spreading}.  To prove this result, we will show that $u(t)$ satisfies a Morawetz-type estimate in the exterior region $|x| \geq t^\beta$ for any $\beta \in (1/3, 1)$.  The advantage of this exterior cutoff is twofold.  First, it allows us to turn the \textit{spatial} decay of $V$ into decay in \textit{time}, making it negligible for the purposes of the estimate.  Second, it eliminates the term 
\begin{equation}\label{eqn:Morawetz-bad-term}
    [{-i}\Delta, \gamma] = {\frac{1}{2}}[{-i}\Delta, [{-i}\Delta, |x|]]
\end{equation}
which would otherwise prevent us from getting any estimate in dimension $1$.  On the other hand, it also introduces the terms 
$$[-\Delta, F(|x| \geq t^\beta)] i\gamma F(|x| \geq t^\beta) +   F(|x| \geq t^\beta) i\gamma [-\Delta, F(|x| \geq t^\beta)]$$
which reflect the fact that mass can flow across the boundary $|x| \sim t^\beta$.  These `quantum correction' terms do not have a sign in general.  However, after carefully symmetrizing, it can be shown that they have the right sign up to an error of size $O(t^{-3\beta})$, and so can be treated perturbatively if $\beta > 1/3$.

Using these Morawetz estimates, we adapt the argument of~\cite{liuLargeTimeAsymptotics2025} to show that $u(t)$ is localized to the region $|x| \leq t^{1/2}$ as $t \to \infty$.  A key novelty here is the use of the additional term in the Morawetz estimate from the defocusing nonlinearity to obtain improved localization for equations with $p < 5$.

To prove~\Cref{thm:nonrad-spreading}, the key observation is that the exterior Morawetz estimate can be used to upgrade the \textit{qualitative} convergence in the nonradiative hypothesis~\eqref{eqn:nonrad-hypo} into a \textit{quantitative} decay rate in time for the `exterior Morawetz current' $\langle F(|x| \geq t^\beta) \gamma F(|x| \geq t^\beta) \rangle_t$.  Since $\gamma$ is the Heisenberg derivative of the observable $|x|$, this decay rate can be integrated to give a bound for $\langle |x| \rangle_t$, and optimizing in $\beta$ gives the bound $\langle |x| \rangle_t \lesssim \jBra{t}^{1/2}$.   Moreover, by using the extra decay coming from the nonlinear portion of the Morawetz estimate, we can get better decay for $\langle F(|x| \geq t^\beta) \gamma F(|x| \geq t^\beta) \rangle_t$ (and hence better localization of $u(t)$) in the \editadd{mass supercritical case} $p < \editadd{5}$.  %This agrees with the intuition that the nonlinearity in~\eqref{eqn:NLS} should be long-range for $p \leq 3$ (see~\cite[Chapter 3]{taoNonlinearDispersiveEquations2006}).

To prove~\Cref{thm:wave-op-and-bdd-state}, we first must show that the free channel wave operator $u_0 \mapsto u_+$ exists.  The idea here is to introduce a projection $\Jfree$ such that
$$\slim_{t \to \infty} (I-\Jfree) e^{it\Delta} = 0$$
which is supported near the propagation set for the free flow in phase space.  A natural choice is to choose $\Jfree$ to localize inside the free flow region $x = 2\xi t + o(t)$, where $\xi$ is the frequency variable.  Having done so, the existence of the free channel is equivalent to showing that the limit
\begin{equation}\label{eqn:scat-slim-intro}
    u_+ = \slim_{t \to \infty} e^{-it\Delta} \Jfree u(t)
\end{equation}
exists.  As explained above, the methods of~\cite{soffer2022large} do not apply here, since the nonlinear term does not obey weighted decay estimates.  To overcome this obstacle, we introduce a new exterior interaction Morawetz estimate, which controls the nonlinearity away from the origin.  Interaction Morawetz estimates were previously introduced in~\cite{collianderGlobalExistenceScattering2004} to prove global well-posedness for defocusing nonlinear Schr\"odinger equations in dimension $3$, and have become a standard tool in proving global well-posedness and scattering.  Obtaining these estimates in lower dimensions is more difficult, since ordinary Morawetz estimates fail due to the singular term~\eqref{eqn:Morawetz-bad-term}.  The first interaction Morawetz estimates in lower dimensions were obtained in~\cite{fangGlobalExistenceRough2007,collianderImprovedInteractionMorawetz2007} in dimension $2$, and hold only locally in time.  Later, these results were improved \editadd{and extended to dimension one} in~\cite{collianderTensorProductsCorrelation2009,planchonBilinearVirialIdentities2009}, yielding the estimate
\begin{equation}\label{eqn:inter-Morawetz-intro}
    \lVert \partial_x |u|^2 \rVert_{L^2_{t,x}}^2 \lesssim \lVert u \rVert_{H^1}^4
\end{equation}
We prove that this estimate holds even in the presence of a nonrepulsive potential provided we localize to $|x| \geq t^{\beta}$ for $\beta \in (1/3,1)$.  This estimate can be interpolated with mass conservation to obtain a pointwise decay estimate in the exterior region, allowing us to prove the existence of the strong limit~\eqref{eqn:scat-slim-intro} defining $u_+$.

To control the remainder $\urem = (I-\Jfree) u(t)$, we note that in addition to the standard Duhamel representation using $u_0$, we also have a representation `from the future' coming from the fact that $e^{-it\Delta} \urem \rightharpoonup 0$.  By combining an incoming/outgoing decomposition similar to the one used in~\cite{taoAsymptoticBehaviorLarge2004b} with some careful stationary phase arguments, we are able to prove that the mass of the solution is concentrated in the region $|x| \lesssim t^{1/2+}$ and that the energy is concentrated in $|x| \lesssim t^{1/3+}$.

%TODO: Write the plan for the rest of the paper

\section{Preliminaries}

{\subsection{Definitions}\label{sec:def}

We use the notation $A \lesssim B$ to mean that $A \leq CB$ for some arbitrary but fixed constant $C>0$.  If $C$ is allowed to depend on the parameters $P_1, P_2, \cdots, P_n$, we write
$A \lesssim_{P_1,P_2,\cdots,P_n} B$.  Similarly, $A \sim B$ means $\frac{1}{C} B \leq A \leq CB$ for some arbitrary but fixed $C > 0$.

In our arguments, we will frequently encounter products of differential operators and multiplication operators $G: \phi(x) \mapsto G(x)\phi(x)$.  To fix notation, we denote by $P(D_x) G(x)$ the operator mapping $\phi$ to $P(D_x) (G(x) \phi(x))$, while $(P(D_x) G(x))$ is the application of $P(D_x)$ to $G(x)$.  (In particular, $(P(D_x) G(x))$ is a multiplication operator.)  Thus, for example, 
$$\partial_x f(x) \phi = \partial_x(f(x) \phi) = (\partial_x f(x))\phi + f(x) \partial_x \phi$$

We now define the smooth cut-off functions $F(x \geq A)$ and $F(x \leq A)$.  Let
\begin{equation*}
    H(x) := \begin{cases}
        e^{-1/x} &  x > 0\\
        0 & x \leq 0
    \end{cases}
\end{equation*}
We define the function $F(x  \editadd{\geq} 1)$ by
\begin{equation*}
    F(x \editadd{\geq} 1) := \begin{cases} 
        0 & x \leq 1/2\\
        \exp\left(-\frac{H(2-x)}{x-1/2}\right) & x > 1/2
    \end{cases}
\end{equation*}
In particular, $F(x \leq 1)$ is smooth.  Moreover, since $G(x) := -\frac{H(2-x)}{x-1/2}$ is strictly increasing for $1/2 < x < 2$, we see that $F$ is strictly increasing on $[1/2,2]$. For every positive real number $A$, we define
\begin{equation}\label{eqn:F-leq-def}\begin{split}
    F(x \geq A) :=& F\left( \frac{x}{A} \geq 1\right)\\
    F(x \leq A) :=& 1 - F(x \geq A)
\end{split}\end{equation}
when differentiating these functions, we will use the notation $F^{(n)}(x \geq A)$ to refer to the $n$th derivative of $F(X \geq 1\editadd{)}$ evaluated at $\frac{x}{A}$, and similarly for $F^{(n)}(x \leq A)$.  \editadd{In particular, this means that
\begin{equation*}
    \partial_x F(x \geq A) = \frac{1}{A} F'(x \geq A)
\end{equation*}}

In the future, we will need to use the fact that $\sqrt{F(x \geq A)}$ and $\sqrt{F'(x \geq A)}$ are smooth, bounded functions.
\begin{lemma}
    For every $A > 0$, the functions $\sqrt{F(x \geq A)}$ and $\sqrt{F'(x \geq A)}$ are smooth.
\end{lemma}
\begin{proof}
    We note that it suffices to prove the result for $A = 1$.  We begin by proving that $\sqrt{F(x \geq 1)}$ is smooth.  From the definition, we see that
    \begin{equation*}\begin{split}
        \partial_x \sqrt{F(x \geq 1)} = \partial_x \sqrt{1 - F(x \leq 1)}
        = \frac{-F'(x \leq 1)}{2 \sqrt{1 - F(x \leq 1)}}
        = \begin{cases}
            0 & x < 1/2\\
            -\frac{G'(x) \exp(G(x))}{2\sqrt{1 - \exp(G(x))}} & x \geq 1/2
        \end{cases}
    \end{split}\end{equation*}
    This expression is smooth for $x < 2$ by the definition of $G$, and vanishes for $x > 2$.  Thus, the only possible singularity is $x = 2$, and the function will be $C^1$ provided that
    \begin{equation*}
        \lim_{x \to 2^-} \partial_x \sqrt{F(x \geq 1)} = 0
    \end{equation*}
    By Taylor expanding the exponentials and using that $G(x) \to 0$ as $x \to 2$, we see that
    \begin{equation*}
        \partial_x \sqrt{F(x \geq 1)} = -\frac{G'(x) (1 + o(1))}{\sqrt{G(x)(1+ o(1))}} = \frac{G'(x)}{\sqrt{G(x)}}(-1+o(1))
    \end{equation*}
    Now, for any $k \geq 0$
    \begin{equation}\label{eqn:G-deriv-expr}
        G^{(k)}(x) = \frac{d^k}{dx^k} \begin{cases}
            0 & x \geq 2\\
            \frac{\exp\left(-\frac{1}{2-x}\right)}{x-1/2} & x < 2
        \end{cases} = \begin{cases}
            0 & x \geq 2\\
            P_k((2-x)^{-1}, (x-1/2)^{-1}) \exp\left(-\frac{1}{2-x}\right) & x < 2\\
        \end{cases}
    \end{equation}
    where $P_k(X,Y)$ is a polynomial in $X$ and $Y$.  In particular, we see that
    $$\lim_{x \to 2^{-}} \frac{G'(x)}{\sqrt{G(x)}} = \lim_{x \to 2^-} P_1((2-x)^{-1}, (x-1/2)^{-1}) \exp\left(-\frac{1}{2(2-x)}\right) = 0$$
    which shows that $\sqrt{F(x \geq 1)}$ is $C^1$.  In general, we have that
    \begin{equation*}
        \frac{d^{k}}{dx^k} \sqrt{F(x \geq 1)} = \sum_{n=1}^k \frac{\mathfrak{F}_{n,k}(x)}{(1 - F(x \leq 1))^{n-1/2}}
    \end{equation*}
    where $\mathfrak{F}_{n,k}(x)$ is of the form
    $$\mathfrak{F}_{n,k}(x) = \sum_{j = 1}^{N_{n,k}} C_{j,n,k}\prod_{i = 1}^{n}F^{(a_{i,j,n,k})}(x \leq 1)$$
    for some $N_{n,k}$, $C_{j,n,k}$, and $a_{i,j,n,k}$ with $1 \leq a_{i,j,n,k} \leq n$.  As in the case $k = 1$, $\sqrt{F(x \geq 1)}$ will be $C^k$ provided that $\lim_{x \to 2^{-}} \frac{d^{k}}{dx^k} \sqrt{F(x \geq 1)} = 0$.  By noting that
    \begin{equation*}
        F^{(k)}(x \leq 1) = \mathfrak{G}_{k}(x) \exp(G(x))
    \end{equation*}
    where $\mathfrak{G}_{n,k}$ is a polynomial in $G'(x), G''(x), \cdots, G^{(k)}(x)$, we see by~\eqref{eqn:G-deriv-expr} that
    \begin{equation*}\begin{split}
         \frac{d^{k}}{dx^k} \sqrt{F(x \geq 1)} =& \sum_{n=1}^k \sum_{j = 1}^{N_{n,k}} \frac{C_{j,n,k}\prod_{i = 1}^{n}F^{(a_{i,j,n,k})}(x \leq 1)}{(1 - F(x \leq 1))^{n-1/2}}\\
         =& \sum_{n=1}^k \sum_{j = 1}^{N_{n,k}} (C_{n,j,k} + o(1)) \mathfrak{P}_{j,n,k}(x) \sqrt{\exp\left(-\frac{1}{2-x}\right)}
    \end{split}\end{equation*}
    where $\mathfrak{P}_{j,n,k}(x)$ is a polynomial in $(2-x)^{-1}$ and $(x-1/2)^{-1}$, and on the last line we have used the Taylor expansion $\exp(G(x)) = 1 + O(G(x))$ as $x \to 2$.  It follows that 
    $$\lim_{x \to 2^-} \frac{d^{k}}{dx^k} \sqrt{F(x \geq 1)} = 0$$
    so $\sqrt{F(x \geq 1)} \in C^k$ for all $k$.

    Turning to $\sqrt{F'(x \geq 1)}$, we observe that since $F'(x \geq 1)$ is nonzero on $(1/2, 2)$, we will have shown that $\sqrt{F'(x \geq 1)} \in C^\infty$ once we show that for each $k \geq 1$,
    $$\lim_{x \to 1/2^+} \frac{d^k}{dx^k} \sqrt{F'(x \geq 1)} = \lim_{x \to 2^-} \frac{d^k}{dx^k} \sqrt{F'(x \geq 1)} = 0$$
    The argument for $x \to 2^-$ is a straightforward modification of the one we give above, while the argument for $x \to 1/2^+$ is similar.
\end{proof}

We also denote by $\tilde{F}_M(x \geq A)$ the function
$$\tilde{F}_M(x \geq A) = F\left(x + a_M \geq b_M A\right)$$
$a_M$ and $b_M$ are chosen so that 
\begin{equation*}
    \frac{M^{-1} - a_M}{b_M} = 1/2,\qquad \frac{M-a_M}{b_M} = 2
\end{equation*}
With these choices, we find that $\supp \tilde{F}_M(x \geq A) = [A/M, \infty)$ and $\supp \tilde{F}'(x \geq A) = [A/M, AM]$.  In particular, by choosing $M$ appropriately, we can guarantee that $\tilde{F}_M(x \geq A)$ is nonzero in any region $|x| \gtrsim A$ and $\tilde{F}'(x \geq A)$ is nonzero in any region $|x| \sim A$ (with fixed but arbitrary implicit constants).  In our work, we will write $\tilde{F}(x \geq A)$ to denote $\tilde{F}_M(x \geq A)$, where $M$ is chosen to guarantee these nonvanishing properties for an appropriate range of $x$ values.}
%TODO: See if I need to define $\tilde{F}$ similarly, and think of a quick way to do this.

\subsection{Dispersive estimates} \label{sec:disp-est}

We will make frequent use of the $L^1 \to L^\infty$ dispersive estimate
\begin{equation}
    \lVert e^{it\Delta} \rVert_{L^1 \to L^\infty} \lesssim |t|^{-1/2}
\end{equation}
As is well known, this estimate implies the Strichartz estimates.  In particular, we have the dual Strichartz estimates:
\begin{thm}\label{thm:dual-Strichartz}
    {For any $\Phi \in L^{q'}_t L^{r'}_x$ with $q', r' \leq 2$ satisfying
    $$\frac{1}{r'}+{\frac{2}{q'}} = \frac{5}{2}$$
    we have that}
    \begin{equation*}
        \left\lVert \int e^{-is\Delta} {\Phi}(x,s)\;ds \right\rVert_{L^2} \lesssim \lVert {\Phi}(x,t) \rVert_{L^{q'}_t L^{r'}_x}
    \end{equation*}
\end{thm}

We now give a number of stationary phase estimates:
\begin{lemma}\label{lem:st-ph-lemma-high-freq}
    For any positive exponents $\alpha$, $\beta$, and $\delta$ satisfying 
    \begin{equation*}
        \delta < \min(1/2, 1-\alpha, 1-\beta)
    \end{equation*}
    {and any $N \geq 0$}, the following inequality holds for $t \geq 1$:
    \begin{equation*}
        \left\lVert F(|x| \leq t^\alpha) e^{it\Delta} F(|D| \geq t^{-\delta}) F(|x| \leq t^\beta) \right\rVert_{L^2 \to L^2} \lesssim_{N,\alpha,\beta,\delta} t^{-N}
    \end{equation*}
\end{lemma}
\begin{proof}
    For $\phi \in L^2$, define the operator $I$ by
    \begin{equation*}\begin{split}
        I(\phi) =& F(|x| \leq t^\alpha) e^{it\Delta} F(|D| \geq t^{-\delta}) F(|x|\leq t^\beta) \phi(x)\\
        =& \frac{F(|x| \leq t^\alpha)}{\sqrt{2\pi}} \int e^{-it\xi^2+ix\xi} F(|D| \geq t^{-\delta}) \hat\psi(\xi)\;d\xi
    \end{split}\end{equation*}
    where $\psi = F(|x| \leq t^\beta)\phi$.  By integrating by parts $K$ times using the identity
    \begin{equation*}
        \frac{i}{2t\xi - x} \partial_\xi e^{-it\xi^2+ix\xi} = e^{-it\xi^2+ix\xi}
    \end{equation*}
    we find that
    \begin{equation*}\begin{split}
        I(\phi) =& \frac{F(|x| \leq t^\alpha)}{\sqrt{2\pi}} \int e^{-it\xi^2+ix\xi} \left(\partial_\xi \frac{i}{2t\xi - x}\right)^K\left(F(|{\xi}| \geq t^{-\delta}) \hat{\psi} \right)\;d\xi\\
        =& \sum_{a + b + c = {K}} C_{a,b,c}  F(|x| \leq t^\alpha) \int e^{-it\xi^2} \frac{t^a}{(2t\xi-x)^{K+a}}  t^{b\delta} F^{{(b)}}(|{\xi}| \geq t^{-\delta}) \partial_\xi^c \hat{\psi} \;d\xi\\
        =:& \sum_{a + b + c = K} c_{a,b,c} F(|x| \leq t^\alpha) I_{a,b,c}(\phi)
    \end{split}\end{equation*}
    for some constants $C_{a,b,c} \in \bbC$.  Now, for $|x| \lesssim t^\alpha$ and $|\xi| \gtrsim t^{-\delta}$, our hypothesis that $\alpha < 1 - \delta$ implies that
    \begin{equation*}
        |2t\xi - x| \gtrsim t|\xi|
    \end{equation*}
    so
    \begin{equation*}\begin{split}
        | I_{a,b,c}({\phi}) | \lesssim& t^{-K} t^{b\delta} \lVert |\xi|^{-K - a}F^{(b)}(|\xi| \geq t^{-\delta})\rVert_{L^2_\xi} \lVert x^c F(|x| \leq t^\beta) \phi \rVert_{L^2_x}\\
        \lesssim& t^{-K} t^{\delta(K + a + b - 1/2)} t^{c\beta} \lVert \phi \rVert_{L^2}\\
        \lesssim& t^{-\delta/2} t^{-K(1 -  \max(2\delta, \delta + \beta))}\lVert \phi \rVert_{L^2}
    \end{split}\end{equation*}
    By H\"older's inequality, we find that
    \begin{equation}
        \lVert F(|x| \leq t^\alpha) I_{a,b,c}(\phi) \rVert_{L^2} \lesssim t^{(\alpha - \delta)/2} t^{-K(1 -  \max(2\delta, \delta + \beta))}\lVert \phi \rVert_{L^2}
    \end{equation}
    Thus, provided $\max(2\delta, \delta + \beta) < 1$, we can choose $K$ sufficiently large that $I(\phi)$ decays at a rate of at least $t^{-N}$ in $L^2$.
\end{proof}

\begin{lemma}\label{lem:st-ph-lemma-low-freq}
    For any positive exponents $\alpha$, $\beta$, and $\delta$ satisfying 
    \begin{equation*}
        \beta > \max(\alpha, \delta, 1-\delta)
    \end{equation*}
    {and any $N \geq 0$,} the following inequality holds for $t \geq 1$:
    \begin{equation*}
        \left\lVert F(|x| \geq t^\beta) e^{it\Delta} F(|D| \leq t^{-\delta}) F(|x| \leq t^\alpha) \right\rVert_{L^2 \to L^2} \lesssim_{N,\alpha,\beta,\delta} t^{-N}
    \end{equation*}
\end{lemma}
\begin{proof}
    As in the proof of~\Cref{lem:st-ph-lemma-high-freq}, we define:
    \begin{equation*}\begin{split}
        I(\phi) =& F(|x| \geq t^\beta) e^{it\Delta} F(|D| \leq t^{-\delta}) F(|x| \leq t^\alpha)\phi\\
        =& \sum_{a + b + c = K} C_{a,b,c}  F(|x| \geq t^\beta) \int e^{-it\xi^2+ix\xi} \frac{t^a}{(2t\xi-x)^{K+a}}  t^{b\delta} \partial_\xi^b F(|D| \leq t^{-\delta}) \partial_\xi^c \hat{\psi} \;d\xi\\
        =:& \sum_{a + b + c = K} C_{a,b,c}  F(|x| \geq t^\beta) I_{a,b,c}(\phi)
    \end{split}\end{equation*}
    where here $\psi = F(|x| \leq t^\alpha)\phi$.  By hypothesis, $\beta > 1-\delta$, so under our microlocalization
    \begin{equation*}
        |2t\xi - x| \gtrsim |x|
    \end{equation*}
    Thus, we find that
    \begin{equation*}\begin{split}
        \lVert F(|x| \geq t^\beta) I_{a,b,c}(\phi) \rVert_{L^2} \lesssim& \left\lVert\frac{t^{a+b\delta}}{|x|^{K+a}}F(|x| \geq t^\beta) \right\rVert_{L^2} \int |\partial_\xi^b F(|D| \leq t^{-\delta}) \partial_\xi^c \hat{\psi}| \;d\xi\\
        \lesssim& t^{a + (b-1/2)\delta -\beta(K + a - 1/2)}\lVert \partial_\xi^c \hat{\psi} \rVert_{L^2}\\
        \lesssim& t^{a + (b-1/2)\delta + c \alpha -\beta(K + a - 1/2)}\lVert \phi \rVert_{L^2}
    \end{split}\end{equation*}
    Substituting into the expression for $I(\phi)$, we find that
    \begin{equation*}
        \lVert I(\phi) \rVert_{L^2} \lesssim_{K,\alpha,\beta,\delta} t^{-\beta(K - 1/2) - \delta/2} t^{\max(1-\beta, \alpha, \delta)K} \lVert \phi \rVert_{L^2}
    \end{equation*}
    {Now, our hypotheses imply that $\beta > \max(\delta, 1-\delta) \geq 1/2$, so $\beta > 1/2 > 1-\beta$, and} the result follows.
\end{proof}

\subsection{Commutator estimates}

Commutator estimates play an important role in proving propagation estimates.  To begin, we record an estimate for commuting physical- and Fourier-space localization:
\begin{lemma}\label{lem:phys-Fourier-comm}
    If $F$ and $G$ are smooth cut-off functions supported near zero, then
    \begin{equation}\label{eqn:phys-Fourier-comm-1}
        [F(|x| \leq A), G(|D| \leq B)] \lesssim \frac{1}{AB}
    \end{equation}
    Moreover, if $\tilde{G}$ is a smooth cut-off function supported in $|x| \geq 1$, then
    \begin{equation}\label{eqn:phys-Fourier-comm-2}
        [F(|x| \leq A), \tilde{G}(|D| \geq B)] \lesssim \frac{1}{AB}
    \end{equation}
\end{lemma}
\begin{proof}
    We begin by proving~\eqref{eqn:phys-Fourier-comm-1}.  Observe that $G(|D| \leq 1)$ can be represented as a convolution:
    \begin{equation*}
        G(|D| \leq 1) \phi(x) = \int {k}(x-y) \phi(y)\;dy
    \end{equation*}
    where {$k = \mathcal{F}^{-1} \editadd{G}(|\xi| \leq 1)$} is smooth and rapidly decaying.  By scaling,
    \begin{equation*}
        G(|D| \leq B) \phi(x) = \int B {k}(B(x-y)) \phi(y)\;dy
    \end{equation*}
    It follows that
    \begin{equation*}\begin{split}
        [F(|x| \leq A), G(|D| \leq B)] \phi =& B\int (F(|x| \leq A) - F(|y| \leq A)) {k}(B(x-y)) \phi(y)\;dy
    \end{split}\end{equation*}
    Thus, by Schur's test,
    \begin{equation*}
        \lVert [F(|x| \leq A), G(|D| \leq B)] \rVert_{L^2 \to L^2} \leq \sup_{x \in \bbR} \lVert B(F(|x| \leq A) - F(|y| \leq A)) {k}(B(x-y)) \rVert_{L^1_y}
    \end{equation*}
    By the Mean Value Theorem,
    $$|F(|x| \leq A) - F(|y| \leq A)| \lesssim \frac{|x-y|}{A} $$
    uniformly in $x,y$, so
    \begin{equation*}\begin{split}
        \lVert [F(|x| \leq A), G(|D| \leq B)] \rVert_{L^2 \to L^2} \lesssim& \frac{1}{A} \sup_x \lVert B|x-y| {k}(B(x-y){)} \rVert_{L^1_y}\\
        \lesssim& \frac{1}{AB}
    \end{split}\end{equation*}
    as required.  Note that $1 - \tilde{G}(|x| \geq B) = G(|x| \leq B)$, so~\eqref{eqn:phys-Fourier-comm-2} follows immediately.
\end{proof}

Using the convolution representation of $F(|D| \leq A)$, we see that for $AB \gg 1$, $F(x > B)F(|D| \leq A)$ will essentially be supported on $F(x > \frac{1}{100}B)$, since $F(|D| \leq A)$ represents a convolution with characteristic length scale $A^{-1}$.  We make this heuristic precise in the following lemma:
\begin{lemma}\label{lem:xd-comm}
    For any $N > 0$,
    \begin{subequations}\begin{align}
        \lVert F(D > B) F(|x| \leq A) F(D \leq \frac{1}{100}B) \rVert_{L^2\to L^2} \lesssim_N& \frac{1}{(AB)^N}\label{eqn:approx-comm-1}\\
        \lVert F(x > B) F(|D| \leq A) F(x \leq \frac{1}{100}B) \rVert_{L^2\to L^2} \lesssim_N& \frac{1}{(AB)^N}\label{eqn:approx-comm-2}
    \end{align}\end{subequations}
\end{lemma}
\begin{proof}
    By the Plancherel theorem, it is enough to prove~\eqref{eqn:approx-comm-2}.  {Arguing as in the proof of~\Cref{lem:phys-Fourier-comm}, we see that this} operator has a kernel representation given by
    \begin{equation*}
        K(x,y) = A F(x > B) k(A(x-y)) F(y \leq \frac{1}{100}B)
    \end{equation*}
    Observe that the localization of $x$ and $y$ implies that $|x-y| \geq cB$ for some constant $c$.  Since $F$ is smooth, $|k(z)| \editadd{\lesssim} |z|^{-N-1}$, so
    \begin{equation*}\begin{split}
        \lVert F(x > B) F(|D| \leq A) F(x \leq \frac{1}{100}B) \rVert_{L^2\to L^2} \leq& (\sup_x \lVert K(x,y) \rVert_{L^1_y})^{1/2}(\sup_y \lVert K(x,y) \rVert_{L^1_z})^{1/2}\\
        \lesssim& \lVert Ak(Az) \rVert_{L^1(|z| \geq cB{)}}\\
        \lesssim& \lVert k(z) \rVert_{L^1(|z| \geq cAB{)}}\\
        \lesssim_N& (AB)^{-N}
    \end{split}\end{equation*}
    as required.
\end{proof}

We also record the following abstract commutator lemmas, which will often be useful.
\begin{lemma}\label{lem:anti-comm}
    {For operators $A,B,C$ with $[A,C] = 0$},
    $$ABC - CBA = A[B,C] {-} C[B,A]$$
\end{lemma}

\begin{lemma}\label{lem:two-term-sym}
 $$A^2B + BA^2 = 2ABA + [A,[A,B]]$$
\end{lemma}

\begin{lemma}\label{lem:three-term-sym}
    Suppose $[A,C] = 0$.  Then,
    \begin{equation}
        A^2BC^2 + C^2 B A^2 = 2(AC)B(AC) + R(A,B,C)
    \end{equation}
    where
    $$R(A,B,C) = A[[A,B],C]C + C[[C,B],A]{A} + A[C,[C,B]]A + C[A,[A,B]]C$$
    involves double commutators of $A$ and $C$ with $B$.
\end{lemma}
{\begin{proof}
    To begin, we observe that
    \begin{equation*}
        A^2 B C^2 = ABC^2A + A[A,B]C^2 = (AC)B(AC) + A[A,B]C^2 + A[B,C] CA
    \end{equation*}
    and reversing the roles of $A$ and $C$ gives
    \begin{equation*}
        \editadd{C^2 B A^2} = \editadd{CBA^2C + C[C,B]A^2} = (AC)B(AC) - C[B,C]A^2 - C[A,B] CA
    \end{equation*}
    Thus, 
    \begin{equation*}\begin{split}
        A^2BC^2 + C^2 BA^2 = 2(AC)B(AC) 
        + \left\{A[A,B]C - C[A,B]A \right\}C 
        + \left\{A[B,C] C - C[B,C]A\right\}A
    \end{split}\end{equation*}
    Considering the first term in curly brackets, we see that
    \begin{equation*}
        A[A,B]C - C[A,B]A = A[[A,B],C] + CA[A,B] - C[[A,B],A] - CA[A,B] = A[[A,B],C] - C[[A,B],A]
    \end{equation*}
    which gives the first and last term of $R(A,B,C)$.  A similar calculation involving the other bracketed term gives the remaining double commutator terms.  
\end{proof}}

\subsection{Other preliminaries}

In our later work, we will make use of the following lemma, which can be seen as a variant of the classical Ladyzhenskaya inequality:
\begin{lemma}\label{lem:GNS-variant}
    If ${G}: [0,\infty) \to [0,\infty)$ is a bounded increasing function, then
    \begin{equation*}
        \lVert |{G}(|x|) \phi|^2 \rVert_{L^\infty} \leq  2\lVert {G}(|x|) \partial_x |\phi| \rVert_{L^2} \lVert {G(|x|) |\phi|} \rVert_{L^2}
    \end{equation*}
\end{lemma}
\begin{proof}
    Let us assume that $x \geq 0$, since the case of negative $x$ is similar.  We have that
    \begin{equation*}\begin{split}
        {G}^2(x) |\phi(x)|^2 =& {G}^2(x) \int_x^\infty \partial_y |\phi(y)|^2\;dy\\
        \leq& \int_x^\infty {G}^2(x) |\partial_y |\phi(y)|^2|\;dy\\
        \leq& 2\int_x^\infty {G}^2(y) |\partial_y \phi(y)| |\phi(y)|\;dy\\
        \leq& 2 \lVert {G(|x|)} \partial_x \phi \rVert_{L^2} \lVert {G(|x|)} \phi \rVert_{L^2}\qedhere
    \end{split}\end{equation*}
\end{proof}

\section{Propagation estimates}\label{sec:PRES}

In order to control the spreading of the solution and take full advantage of the defocusing character of the nonlinearity, we will use several propagation estimates.  Recall that we define the expectation of $B$ with respect to $u$ at time $t$ to be
\begin{equation*}
    \langle B \rangle_{t,u} = \langle Bu, u \rangle_{L^2}
\end{equation*}
and will write $\langle B \rangle_t$ when the function $u$ is clear from context.  For a time-dependent operator $B(t)$, the evolution of the expectation \editadd{of a solution $u$ to~\eqref{eqn:NLS}} is given by
\begin{equation*}
    \partial_t \langle B(t) \rangle_t = \langle [-i\Delta + iV + i|u|^{p-1}, B(t)] + \partial_t B(t) \rangle_t = \langle D_H B(t) \rangle_t
\end{equation*}
In our work, the Morawetz operator
\begin{equation*}
    \gamma = \frac{\sgn(x) \partial_x + \partial_x \sgn(x)}{2i}
\end{equation*}
will play an important role.  {If we formally commute the derivative through the nonsmooth function $\sgn(x)$, we have that}
\begin{equation*}
    {\gamma = -i\sgn(x) \partial_x - i \delta(x)}
\end{equation*}
{If $G(x)$ is a function supported away from $x=0$, then the singularity at $x = 0$ does not contribute, and we have the identities }
\begin{equation}\label{eqn:gamma-identities}
    \begin{split}
    {\gamma G =}& {-i\sgn(x) \partial_x G}\\
    {G \gamma  =}& {-i G \sgn(x) \partial_x}\\
    {[\gamma, G] =}& {-i\sgn(x) \editadd{\partial_x G}(x)}
    \end{split}
\end{equation}
{In particular,}
\begin{equation}\label{eqn:gamma-G-abs-comm}
    {[-i\Delta, G(|x|)] = -i \partial_x \sgn(x) G'(|x|) -i G'(|x|)\sgn(x) \partial_x = \gamma G'(|x|) + G'(|x|) \gamma}
\end{equation}

For Schr\"odinger equations with purely repulsive interactions, $\langle D_H \gamma \rangle_t$ has good positivity properties (at least in dimension greater than $3$), which allows us to control the solution in certain global space-time norms.  Since we do not assume that the potential in~\eqref{eqn:NLS} is repulsive, we cannot expect to derive estimates using $\langle D_H \gamma \rangle_t$.  However, we are able to get estimates in the exterior of a time-dependent region, where the potential can be treated as an integrable perturbation.  

\subsection{Basic estimates}
Here, we collect some basic estimates, which we will refer to later in~\Cref{sec:loc-nonrad}.  We begin with two estimates on $\gamma u$ in the region $|x| \sim t^\beta$:
\begin{lemma}\label{lem:gamma-annulus-int}
    For $\beta \in (1/3, 1)$, we have that
    \begin{equation}\label{eqn:gamma-annulus-int-0}
        \left|\int_0^T \langle F'(|x| \geq t^\beta) \gamma F'(|x| \geq t^\beta) \rangle_t \;dt\right| \lesssim T^\beta \editadd{\lVert u \rVert_{L^2}^2}
    \end{equation}
    {and}
    \begin{equation}\label{eqn:gamma-annulus-int-infty}
        \sup_{T \in (t_0, \infty]}{\left| \int_{t_0}^T \frac{1}{t}\langle F'(|x| \geq t^\beta) \gamma F'(|x| \geq t^\beta) \rangle_t \;dt \right| \lesssim t_0^{\beta-1}} \editadd{\lVert u \rVert_{L^2}^2}
    \end{equation}
\end{lemma}
\begin{proof}
    {We begin by proving~\eqref{eqn:gamma-annulus-int-0}.}  Let $G$ be the bounded function given by $G(0) = 0$ and ${G'(x) = (F'(x))^2}$.  Then, by~\eqref{eqn:Heis-deriv},
    \begin{equation}\label{eqn:t-alpha-g-pres}\begin{split}
        \langle T^\beta G({T^{-\beta}}|x| {)} \rangle_T =& \int_0^T \partial_t \langle t^\beta G({t^{-\beta}}|x| {)} \rangle_t\; dt\\
        {=}& {\int_0^T \langle D_H \left( t^\beta G(t^{-\beta} |x|)\right)\rangle_t\;dt}\\
        =& \int_0^T \langle \gamma G' + G' \gamma \rangle_t + \beta \langle t^{\beta - 1} G \rangle_t {-\beta} \langle \frac{|x|}{t} G' \rangle_t \;dt
    \end{split}\end{equation}
    {where on the last line we have used~\eqref{eqn:gamma-G-abs-comm}.  For notational simplicity, we will suppress the argument $(t^{-\beta} |x|)$ of the functions $G$ and $G'$.}  By the definition of $G'$ and~\Cref{lem:two-term-sym},
    \begin{equation}\label{eqn:G'-to-F'-eqn}\begin{split}
        \langle G' \gamma + \gamma G' \rangle_t =& \langle (F')^2 \gamma + \gamma (F')^2 \rangle_t\\
        =& 2 \langle F' \gamma F' \rangle_t + \langle [F', [F', \gamma]]\rangle_t\\
        =& 2 \langle F' \gamma F' \rangle_t
    \end{split}\end{equation}
    \editadd{Here, the $[F', [F', \gamma]]$ term drops because by~\eqref{eqn:gamma-identities},}
    $${[F',\gamma] = -[\gamma, F']= i t^{-\beta} F''}$$
    \editadd{is a function, so $[F', [F', \gamma]] = 0$.}  Inserting this into~\eqref{eqn:t-alpha-g-pres} and rearranging, we see that
    \begin{equation*}
        \int_0^T 2 \langle F' \gamma F' \rangle_t\;dt = \langle T^\beta G( T^{-\beta}|x|) \rangle_T + \int_0^T O(t^{\beta - 1} \editadd{\lVert u \rVert_{L^2}^2}) \;dt
    \end{equation*}
    which gives~{\eqref{eqn:gamma-annulus-int-0}}.

    {Turning to~\eqref{eqn:gamma-annulus-int-infty}, we apply~\eqref{eqn:Heis-deriv} to the propagation observable $t^{\beta-1} G(t^{-\beta} |x|)$ to find that}
    \begin{equation*}\begin{split}
        {\langle T^{\beta-1} G(T^{-\beta} |x|) \rangle_T - \langle t_0^{\beta-1} G(t_0^{-\beta} |x|) \rangle_{t_0} \lesssim}& {\int_{t_0}^T \frac{1}{t}\langle \gamma G'(t^{-\beta}|x|) + G'(t^{-\beta}|x|)\gamma\rangle_t\;dt}\\ &{-\int_{t_0}^t(\beta-1) t^{\beta-2} \langle G(t^{-\beta}|x|) \rangle_t + \beta t^{\beta-2} \left\langle \frac{|x|}{t^\beta} G'(t^{-\beta}|x|)\right\rangle_t\;dt}
    \end{split}\end{equation*}
    {Using~\eqref{eqn:G'-to-F'-eqn} and rearranging, we find that}
    \begin{equation*}\begin{split}
        {\int_{t_0}^T \frac{2}{t}\langle F'(|x| \geq t^\beta) \gamma F'(|x| \geq t^\beta) \rangle_t \;dt =}& {\langle T^{\beta-1} G(T^{-\beta} |x|) \rangle_{T}-\langle t_0^{\beta-1} G(t_0^{-\beta} |x|) \rangle_{t_0}}\\
        &{+ \int_{t_0}^t(\beta-1) t^{\beta-2} \langle G(t^{-\beta}|x|) \rangle_t + \beta t^{\beta-2} \left\langle \frac{|x|}{t^\beta} G'(t^{-\beta}|x|)\right\rangle_t\;dt}\\
        {=}& {O(t_0^{\beta-1})\editadd{\lVert u \rVert_{L^2}^2}}
    \end{split}\end{equation*}
    {which is~\eqref{eqn:gamma-annulus-int-infty}.}
\end{proof}

\begin{rmk}\label{rmk:F-to-tilde-F}
    {The only property of $F$ used in the above proof is the fact that $F'(|x| \geq t^{\beta})$ is supported in $|x| \sim t^\beta$.  It follows that the results continue to apply if we replace $F'$ by $\tilde{F}'$, which will prove useful in several of the following arguments.}
\end{rmk}

\subsection{An exterior Morawetz estimate}
We now give an exterior Morawetz estimate, which we will use in~\Cref{sec:loc-nonrad} to obtain slow spreading for nonradiative states that are initially localized.  A similar result without the nonlinear term was proved in~\cite{liuLargeTimeAsymptotics2025}.

\begin{thm}\label{thm:ext-mora-est}
    \editadd{For any $\beta \in (1/3, 1)$,} the following estimate holds for solutions to~\eqref{eqn:NLS}:
    \begin{equation}\label{eqn:ext-mora-diff-est}\begin{split}
        \langle F \gamma F \rangle_T - \langle F \gamma F \rangle_t =& \int_t^T 2s^{-\beta} \lVert \gamma \sqrt{F'F} u(s) \rVert_{L^2}^2 + {2}s^{-\beta}\frac{p-1}{p+1} \int F'F |u(s)|^{p+1}\;dx\;ds\\
        &+ \int_{t}^{\editadd{T}} O(s^{-3\beta}) \langle {G(x/s^\beta)} \rangle_s\; ds + O(t^{1 - \beta(\sigma+1)} {+ t^{\beta - 1}}) \lVert u_0 \rVert_{L^2}^2
    \end{split}\end{equation}
    where ${G(x)}$ is a bounded function supported in the region $|x| \sim 1$.  In particular, if $\sigma {\geq} 2$, then 
    \begin{equation}\label{eqn:ext-mora-usual-est}
        \int_1^\infty \left(2s^{-\beta} \lVert \gamma \sqrt{F'F} u \rVert_{L^2}^2 + {2}s^{-\beta}\frac{p-1}{p+1} \int F'F |u|^{p+1}\;dx\right)\;ds \editadd{\lesssim \lVert u \rVert_{L^2} \sup_t\lVert u(t) \rVert_{H^1}}
    \end{equation}
\end{thm}
%To prove this theorem, we will need the following lemma:
%\begin{lemma}\label{thm:ext-mora-est-lemma}
%    If $\sigma {\geq} 2$, $\beta > \frac{1}{3}$, then
%    \begin{equation*}
%        t^{-\beta - \epsilon} \lVert \gamma \sqrt{{\tilde{F}}'{\tilde{F}}} u \rVert_{L^2}^2 \in L^1(1,\infty)
%    \end{equation*}
%\end{lemma}
\begin{proof}
    We begin by considering the propagation observable 
    ${F(|x| \geq t^\beta) \gamma F(|x| \geq t^\beta)}$.  Taking the Heisenberg derivative, we find that
    \begin{equation}\label{eqn:B-deriv}\begin{split}
        \partial_t \langle F(|x| \geq t^\beta) \gamma F(|x| \geq t^\beta) \rangle_t =& \langle [-i\Delta, F\gamma F] \rangle_t + \langle [iV, F\gamma F] \rangle_t + \langle [i|u|^{p-1}, F\gamma F] \rangle_t\\
        &\editadd{-\beta} \left\langle \frac{|x|}{t^{\beta + 1}} F' \gamma F + F \gamma F'\frac{|x|}{t^{\beta + 1}} \right\rangle_t
    \end{split}\end{equation}
    {For the first term, we note that by~\eqref{eqn:gamma-G-abs-comm} and the commutator identity from~\eqref{eqn:gamma-identities},}
    \begin{equation}\label{eqn:Lap-comm-expansion}
        {[-i\Delta, F] = t^{-\beta} \gamma {{F}}' + t^{-\beta} {{F}}' \gamma = 2 t^{-\beta}{{F}}' \gamma - it^{-2\beta}{{F}}'' = 2 t^{-\beta} \gamma {{F}}' {+} it^{-2\beta}{{F}}''}
    \end{equation}
    {while $F[-i\Delta, \gamma]F = 0$, so we can write}
    \begin{equation*}
        {[-i\Delta, F\gamma F] = 2t^{-\beta} (F' \gamma^2 F + F \gamma^2 F') - it^{-2\beta} (F'' \gamma {\tilde{F}} - F\gamma F'')}
    \end{equation*}
    {Applying the commutator identities of~\Cref{lem:anti-comm,lem:three-term-sym} and noting that each commutator between $\gamma$ and $\sqrt{F}$, $\sqrt{F'}$ or their derivatives gains us a factor of $t^{-\beta}$, we find that}
    \begin{equation*}
        {[-i\Delta, F\gamma F] = 2 t^{-\beta} \sqrt{F'F}\gamma^2 \sqrt{F'F} + t^{-3\beta} G(x / t^\beta)}
    \end{equation*}
    {where $G(x)$ is some bounded function supported in the region $|x| \sim 1$ (so $G(x/t^\beta)$ is supported in $|x| \sim t^\beta$).  Turning to the potential term, we see that}
    \begin{equation*}
        {[iV, F\gamma F] = -F V' \sgn(x) F}
    \end{equation*}
    {so by the decay hypothesis~\eqref{eqn:dV-decay-hypo} we have that}
    \begin{equation*}
        {\langle [iV, F\gamma F] \rangle_t \lesssim t^{-\beta(\sigma +1)} \lVert u \rVert_{L^2}^2}
    \end{equation*}
    {For the nonlinear term, a similar calculation gives that}
    \begin{equation*}
        {[i|u|^{p-1}, F\gamma F] = -F (\partial_x |u|^{p-1}) \sgn(x) F}
    \end{equation*}
    {and we can integrate by parts to find that}
    \begin{equation*}\begin{split}
        {\langle [i|u|^{p-1}, F\gamma F] \rangle_t =}& {-\int |u|^2(\partial_x |u|^{p-1}) F^2 \sgn(x) \;dx}\\
        {=}& {-\frac{p-1}{p+1}\int (\partial_x |u|^{p+1}) F^2 \sgn(x) \;dx}\\
        {=}& {2t^{-\beta} \frac{p-1}{p+1} \int F'F |u|^{p+1}\;dx}
    \end{split}\end{equation*}
    {where on the last line we have used the fact that}
    $${\partial_x {{F}} = \partial_x {{F}}(|x| \geq t^{\beta}) = t^{-\beta}\sgn(x) {{F}}'}$$
    Thus, it only remains to consider the last term in~\eqref{eqn:B-deriv}.  Applying~\Cref{lem:three-term-sym} with $A = \sqrt{|x| t^{-\beta} F'}$, $B = \gamma$, and $C = \sqrt{F}$, we see that
    \begin{equation*}
        \frac{|x|}{t^{\beta + 1}} F' \gamma F + F \gamma F'\frac{|x|}{t^{\beta + 1}}  = \frac{2}{t} \sqrt{|x|t^{-\beta} F'F} \gamma \sqrt{|x|t^{-\beta} F'F}
    \end{equation*}
    {Notice that since $\gamma$ is a first order differential operator, there are no double commutator terms.  To control this term, we will prove a propagation estimate analogous to~\Cref{lem:gamma-annulus-int}.  Define $H(x)$ with $H(0) = 0$ and $H'(y) = yF'(y \geq 1)F(y \geq 1)$.  Noting that $\lim_{T \to \infty} \langle T^{\beta -1} H(T^{-\beta} |x|)\rangle_{T} = 0$, we see that}
    \begin{equation*}\begin{split}
        {\langle t^{\beta - 1} H(t^{-\beta} |x|) \rangle_t =}& {-\int_t^\infty \frac{1}{s} \langle \gamma H'(s^{\editadd{-}\beta}|x|) + H'(s^{\editadd{-}\beta}|x|) \gamma\rangle_s \;ds}\\
        &{+ \int_t^\infty (\beta - 1) s^{\beta - 2} \langle H(s^{-\beta}|x|)\rangle_s \editadd{-} \beta s^{\beta - 2} \left\langle \frac{|x|}{s^\beta}H'(s^{-\beta}|x|)\right\rangle_s\;ds}
    \end{split}\end{equation*}
    By the reasoning used in~\eqref{eqn:G'-to-F'-eqn}, we have that $H'\gamma + \gamma H' = 2\sqrt{|x|t^{-\beta} F'F} \gamma \sqrt{|x|t^{-\beta} F'F}$, \editadd{and the integrand on the second line can be seen to be $O(s^{\beta - 2}\lVert u \rVert_{L^2}^2)$.}  Thus, rearranging gives that
    \begin{equation*}\begin{split}
        \int_t^\infty \frac{2}{s} \langle\sqrt{|x|\editadd{s}^{-\beta} F'F} \gamma \sqrt{|x|\editadd{s}^{-\beta} F'F}\rangle_s \;ds \editadd{=}&
        \editadd{-\langle t^{\beta - 1} H(t^{-\beta} |x|) \rangle_t \rangle_t + \int_t^\infty O(s^{\beta - 2} \lVert u \rVert_{L^2}^2)\;ds}\\
        =& O(t^{\beta-1} \editadd{\lVert u \rVert_{L^2}^2})
    \end{split}\end{equation*}
    
    Collecting the previous estimates, we see that
    \begin{equation*}\begin{split}
        \langle F \gamma F \rangle_T - \langle F \gamma F \rangle_t =& \int_t^T 2s^{-\beta} \lVert \gamma \sqrt{F'F} u \rVert_{L^2}^2 + {2}s^{-\beta}\frac{p-1}{p+1} \int F'F |u|^{p+1}\;dx\;ds\\
        +& \int_{t}^\infty O(s^{-3\beta}) \langle {G(x/t^\beta)} \rangle_s\; ds + O(t^{1 - \beta(\sigma+1)} + t^{\beta - 1})\lVert u \rVert_{L^2}^2
    \end{split}\end{equation*}
    which proves~\eqref{eqn:ext-mora-diff-est}.  The bound~\eqref{eqn:ext-mora-usual-est} then follows immediately by taking $t = 1$ and $T \to \infty$.
\end{proof}

For future use, we observe that the above argument still goes through if we replace $F$ by $\tilde{F}$ everywhere.  In particular,
{\begin{equation}\label{eqn:tilde-ext-mora-usual-est}
        \int_1^\infty \left(2s^{-\beta} \lVert \gamma \sqrt{\tilde{F}'\tilde{F}} u \rVert_{L^2}^2 + {2}s^{-\beta}\frac{p-1}{p+1} \int \tilde{F}'\tilde{F} |u|^{p+1}\;dx\right)\;ds \editadd{\lesssim \lVert u \rVert_{L^2} \sup_t\lVert u(t) \rVert_{H^1}}
\end{equation}}

\subsection{An exterior interaction Morawetz estimate}
Since~\eqref{eqn:NLS} admits nondispersive solutions for general potentials $V$, we cannot expect the usual interaction Morawetz estimates to hold.  However, we will show that $u$ obeys an interaction Morawetz estimate when restricted to the exterior of a time-dependent region:
\begin{thm}\label{thm:ext-int-Mora}
    The solution $u$ to~\eqref{eqn:NLS} satisfies
    \begin{equation}\label{eqn:inter-Mor}
         \int_1^\infty \lVert \editadd{\left(1-F^2(|x| \leq t^{\beta})\right)} \partial_x |u|^2 \rVert_{L^2_{x}}^2\;dt + \int_1^\infty \lVert \editadd{\left(1-F^2(|x| \leq t^{\beta})\right)}^{\frac{2}{p+3}} u \rVert_{L^{p+3}_x}^{p+3}\;dt \lesssim \sup_{t \in [1,\infty)} \lVert u(t) \rVert_{L^2}^3 \lVert u(t) \rVert_{H^1}
    \end{equation}
\end{thm}
We will prove~\Cref{thm:ext-int-Mora} by adapting the tensor argument of~\cite{planchonBilinearVirialIdentities2009}.  Define $U(x,y,t) = u(x,t)u(y,t)$.  Then, $U$ solves the two-dimensional Schr\"odinger equation
\begin{equation}\label{eqn:tensor-NLS}
    i \partial_t U + \Delta_{x,y} U = (V(x,t) + V(y,t))U + (|u(x,t)|^{p-1} + |u(y,t)|^{p-1})U
\end{equation}
where $\Delta_{x,y} = \Delta_x + \Delta_y$.  We will now derive a propagation estimate for $U$ which implies the estimate~\eqref{eqn:inter-Mor} for $u$.  
We define the interaction Morawetz vector fields to be
\editadd{\begin{equation*}
    \gamma^\otimes_+ := -i\bbOne_{x > y} \partial_x -i\bbOne_{x < y} \partial_y,\qquad\qquad \gamma^\otimes_- = i\bbOne_{x > y} \partial_y +i \bbOne_{x < y} \partial_x
\end{equation*}
    We remark that, due to boundary terms, the above operators are not self-adjoint.  Although we could add the boundary terms in `by hand,' this would needlessly complicate later calculations.  Instead, we observe that for any operator $A$,
    \begin{equation*}
        \Re \langle A U, U \rangle = \langle \frac{1}{2}(A + A^*) U, U \rangle
    \end{equation*}
    so we may achieve the same result by taking the real part in all propagation estimates involving $\gamma^\otimes_\pm$.  We remark that (up to boundary terms), the proof of the interaction Morawetz estimate in~\cite{planchonBilinearVirialIdentities2009} follows from taking propagation estimates using the operator $\gamma^\otimes_+ + \gamma^\otimes_-$.  The following proposition shows that in fact, either of the vector fields $\gamma^\otimes_\pm$ serves to give the same estimate:
\begin{prop}\label{prop:gamma-otimes-ident}
    Let $\mathfrak{U} = \mathfrak{u}(x) \mathfrak{v}(y)$ some $H^1$ functions $\fu$ and $\fv$.  Then, we have that
    \begin{equation}\label{eqn:gamma-otimes-ident}
        \Re \langle i[-\Delta_{x,y}, \gamma^\otimes_\pm] \mathfrak{U}, \mathfrak{U} \rangle \geq \int_{-\infty}^\infty \partial_x |\mathfrak{u}(x,t)|^2 \partial_x |\mathfrak{v}(x,t)|^2\;dx
    \end{equation}
\end{prop}
\begin{proof}
    At its heart, this is nothing more than a positive commutator estimate.  However, the form of $\gamma^\otimes$ leads to a number of boundary terms when performing integrations by parts, so we will proceed carefully.  We will focus on the argument for $\gamma^\otimes_+$; the argument for the other vector field is similar.  Using the definition of $\gamma^\otimes$, we have that
    \begin{subequations}\begin{align}
        \langle i[-\Delta_{x,y}, \gamma^\otimes_\pm] \mathfrak{U}, \mathfrak{U} \rangle =& \langle i \gamma^\otimes \mathfrak{U},  (-\Delta_{x,y} \mathfrak{U}) \rangle - \langle  i \gamma^\otimes_+(-\Delta_{x,y}) \mathfrak{U},  \mathfrak{U} \rangle\notag\\
            =& -\int_{x > y}  \partial_x\overline{\mathfrak{U}} \partial_x^2 \mathfrak{U} - \mathfrak{U}\partial_x^3\overline{\mathfrak{U}}  \;dxdy\label{eqn:EIM-free-1}\\
            & -\int_{x > y}  \partial_x \overline{\mathfrak{U}} \partial_y^2\mathfrak{U} - \partial_x\partial_y^2 \overline{\mathfrak{U}}  \mathfrak{U}\;dxdy\label{eqn:EIM-free-2}\\
            & -\int_{x < y}  \partial_y\overline{\mathfrak{U}}  \partial_x^2\mathfrak{U} - \partial_x^2 \partial_y \overline{\mathfrak{U}}  \mathfrak{U}\;dxdy\label{eqn:EIM-free-3}\\
            & -\int_{x < y}  \partial_y \overline{\mathfrak{U}} \partial_y^2\mathfrak{U} - \mathfrak{U}\partial_y^3 \overline{\mathfrak{U}}  \;dxdy\label{eqn:EIM-free-4}
    \end{align}
    \end{subequations}
    For~\eqref{eqn:EIM-free-1}, we compute that
    \begin{equation*}\begin{split}
        \eqref{eqn:EIM-free-1} =& -\iint_{x > y} \partial_x |\partial_x \mathfrak{U}|^2\;dxdy  - \int_{-\infty}^\infty \mathfrak{U}\partial_x^2\overline{\mathfrak{U}}(y, y) \;dy
                                = \int_{-\infty}^\infty |\partial_x \mathfrak{U}|^2(y, y) - \mathfrak{U}\partial_x^2\overline{\mathfrak{U}}  \;dy
    \end{split}\end{equation*}
    By symmetry, we also have that
    \begin{equation*}
        \eqref{eqn:EIM-free-4} = \int_{-\infty}^\infty |\partial_y \mathfrak{U}|^2(x, x) - \mathfrak{U}\partial_y^2 \overline{\mathfrak{U}} (x, x) \;dx
    \end{equation*}
    Similarly, for~\eqref{eqn:EIM-free-2}, integrating by parts in $y$ gives
    \begin{equation*}\begin{split}
        \eqref{eqn:EIM-free-2} =& \int_{-\infty}^\infty \mathfrak{U} \partial_x\partial_y \overline{\mathfrak{U}}(x,x) - \partial_y \mathfrak{U}\partial_x \overline{\mathfrak{U}} (x,x)\;dx
    \end{split}\end{equation*}
    and, by symmetry,
    \begin{equation*}\begin{split}
        \eqref{eqn:EIM-free-3} 
        =& \int_{-\infty}^\infty \mathfrak{U} \partial_x\partial_y \overline{\mathfrak{U}}(y,y) - \partial_x \mathfrak{U} \partial_y \overline{\mathfrak{U}}(y,y)\;dy
    \end{split}\end{equation*}
    In particular, summing these equalities and using the fact that $\mathfrak{U} = \fu\fv$ near the diagonal gives
    \begin{equation*}\begin{split}
        \langle\mathfrak{U}, i[-\Delta_{x,y}, \gamma^\otimes_+] \mathfrak{U} \rangle =& \eqref{eqn:EIM-free-1} +\eqref{eqn:EIM-free-2} +\eqref{eqn:EIM-free-3} +\eqref{eqn:EIM-free-4}\\
        =& \int_{-\infty}^\infty |\fu'(r)|^2 |\fv(r)|^2 + |\fu(r)|^2 |\fv'(r)|^2 - {\fu}(r) \overline{\fu''}(r)|\fv(r)|^2 - |\fu(r)|^2 {\fv}(r) \overline{\fv''}(r)\;dr\\
        &+ 2\int_{-\infty}^\infty {\fu}(r) {\fv}(r) \overline{\fu'}(r) \overline{\fv'}(r) -  \Re ({\fu'}(r) \overline{\fu}(r) {\fv}(r) \overline{\fv'}(r))\;dr\\
        =& \int_{-\infty}^\infty 2|\fu'(r)|^2 |\fv(r)|^2 + 2|\fu(r)|^2 |\fv'(r)|^2 + {\fu}(r) \overline{\fu'(r)} \partial_r|\fv(r)|^2 + \partial_r |\fu(r)|^2 {\fv}(r) \overline{\fv'}(r)\;dr\\
        &+ 2\int_{-\infty}^\infty {\fu}(r) {\fv}(r) \overline{\fu'}(r) \overline{\fv'}(r) -  \Re ({\fu'}(r) \overline{\fu}(r) {\fv}(r) \overline{\fv'}(r))\;dr
    \end{split}\end{equation*}
    Taking the real part now gives
    \begin{equation*}\begin{split}
        \Re\langle\mathfrak{U}, i[-\Delta_{x,y}, \gamma^\otimes_+] \mathfrak{U} \rangle =& \int_{-\infty}^\infty \partial_r |\fv(r)|^2 \partial_r|\fv(r)|^2 \;dr + 2\int_{-\infty}^\infty |\fu'(r)|^2 |\fv(r)|^2 + |\fu(r)|^2 |\fv'(r)|^2 \;dr\\
        %&\\
        &+ 2\int_{-\infty}^\infty \Re(\overline{\fu}(r) \overline{\fv}(r) \fu'(r) \fv'(r)) -  \Re (\overline{\fu'}(r) \fu(r) \overline{\fv}(r) \fv'(r))\;dr\\
        \geq& \int_{-\infty}^\infty \partial_r |\fv(r)|^2 \partial_r|\fv(r)|^2 \;dr
    \end{split}\end{equation*}
    where on the last line we have used the elementary inequality
    \begin{equation*}
        2|\Re(ab)| \leq |a|^2 + |b|^2
    \end{equation*}
    to control the last integral.
\end{proof}
}

\editadd{We are now in a position to prove~\Cref{thm:ext-int-Mora}:
\begin{proof}[Proof of~\Cref{thm:ext-int-Mora}]
    Define the cutoff functions
    \begin{equation*}
        F^\otimes_+(x,y) = 1- F(x \leq t^\beta) F(y \leq t^\beta),\qquad F^\otimes_-(x,y) = 1- F(-x \leq t^\beta) F(-y \leq t^\beta)
    \end{equation*}
    \Cref{thm:ext-int-Mora} follows from a propagation estimate against the observable $F^\otimes_+ \gamma^\otimes_+ F^\otimes_+ + F^\otimes_- \gamma^\otimes_- F^\otimes_-$.  We will show how to obtain the estimate for $F^\otimes_+ \gamma^\otimes_+ F^\otimes_+$ (which gives~\eqref{eqn:inter-Mor} for $x > 0$); the estimate for $F^\otimes_- \gamma^\otimes_- F^\otimes_-$ is similar.  

    Using~\eqref{eqn:tensor-NLS}, we compute that
    \begin{subequations}\begin{align}
        \partial_t \langle F^\otimes_+ \gamma^\otimes_+ F^\otimes_+ \rangle_{t,U} =& \langle i[-\Delta_{x,y}, F^\otimes_+ \gamma^\otimes_+ F^\otimes_+] \rangle_{t,U}\label{eqn:EIM-linear}\\
        &+ \langle i[V(x,t) + V(y,t), F^\otimes_+ \gamma^\otimes_+ F^\otimes_+] \rangle_{t,U}\label{eqn:EIM-pot}\\
        &+ \langle i[|u(x,t)|^{p-1} + |u(y,t)|^{p-1}, F^\otimes_+ \gamma^\otimes_+ F^\otimes_+] \rangle_{t,U}\label{eqn:EIM-nonlin}\\
        &+ \langle \partial_t F^\otimes_+ \gamma^\otimes_+ F^\otimes_+ + F^\otimes_+ \gamma^\otimes_+ \partial_t F^\otimes_+ \rangle_{t,U}\label{eqn:EIM-dt}
    \end{align}
    \end{subequations}
    Here, the terms~\eqref{eqn:EIM-dt} and~\eqref{eqn:EIM-pot} will be purely perturbative, while the terms~\eqref{eqn:EIM-linear} and~\eqref{eqn:EIM-nonlin} contain perturbative and positive terms.  For~\eqref{eqn:EIM-pot}, the support condition on $F^\otimes_+$ and~\eqref{eqn:dV-decay-hypo} immediately imply that
    \begin{equation*}
        |\eqref{eqn:EIM-pot}| \lesssim t^{-(\sigma+1)\beta}\lVert u \rVert_{L^2}^4
    \end{equation*}
    Turning to~\eqref{eqn:EIM-dt}, we compute that
    \begin{equation*}
        \partial_t F^\otimes_+ = \beta \frac{x}{t^{\beta+1}} F'(x \leq t^\beta) F(y \leq t^\beta) + \beta \frac{y}{t^{\beta+1}} F(x \leq t^\beta) F'(y \leq t^\beta)
    \end{equation*}
    Observe that for $x > y$, $\gamma^\otimes_+ = \gamma_x$ (the $\gamma$ vector field in the $x$ variable) and $\partial_t F^\otimes_+$ is supported in the region $x \sim t^\beta$; while for $x < y$, $\gamma^\otimes_+ = \gamma_y$ and  $\partial_t F^\otimes_+$ is supported in the region $y \sim t^\beta$.  Thus, we have that
    \begin{equation*}
        |\eqref{eqn:EIM-dt}| \lesssim t^{-1}\lVert u \rVert_{L^2}^3 \left\lVert \gamma \sqrt{\tilde{F}'\tilde{F}(|x| \geq t^\beta)} u \right\rVert_{L^2}
    \end{equation*}
    so, using~\Cref{thm:ext-mora-est} and recalling that $\beta < 1$, we have that for any $T \in (1,\infty)$,
    \begin{equation*}
        \int_1^T |\eqref{eqn:EIM-dt}| \;dt \lesssim \int_1^T t^{-1+\beta/2} t^{-\beta/2} \left\lVert \gamma \sqrt{\tilde{F}'\tilde{F}(|x| \geq t^\beta)} u \right\rVert_{L^2} \lVert u \rVert_{L^2}^3\;dt \lesssim \sup_t \lVert u \rVert_{L^2}^3 \lVert u \rVert_{H^1}
    \end{equation*}
    For the nonlinear term~\eqref{eqn:EIM-nonlin}, we write
    \begin{equation*}\begin{split}
        \eqref{eqn:EIM-nonlin}  =& -\left\langle F^\otimes_+ \left(\bbOne_{x > y} \partial_x |u(x,t)|^{p-1} + \bbOne_{x < y} \partial_y |u(y,t)|^{p-1}\right) F^\otimes_+  \right\rangle_{t,U}\\
                                =& - \iint_{x > y} \left(F^\otimes_+\right)^2 |u(x,t)|^2|u(y,t)|^2 \partial_x |u(x,t)|^{p-1} \;dxdy\\
                                 & - \iint_{x < y} \left(F^\otimes_+\right)^2 |u(x,t)|^2|u(y,t)|^2 \partial_y |u(y,t)|^{p-1} \;dxdy\\
                                =& - \frac{p-1}{p+1}\iint_{x > y} \left(F^\otimes_+\right)^2 |u(y,t)|^2 \partial_x |u(x,t)|^{p+1} \;dxdy\\
                                 & - \frac{p-1}{p+1}\iint_{x < y} \left(F^\otimes_+\right)^2 |u(x,t)|^2 \partial_y |u(y,t)|^{p+1} \;dxdy\\
                                =& \frac{p-1}{p+1} \iint_{x > y} \partial_x \left(F^\otimes_+\right)^2 |u(x,t)|^{p+1} |u(y,t)|^2\;dxdy\\
                                 &+ \frac{p-1}{p+1} \iint_{x < y} \partial_y \left(F^\otimes_+\right)^2 |u(x,t)|^{2} |u(y,t)|^{p+1}\;dxdy\\
                                 &+ 2\frac{p-1}{p+1} \int_{-\infty}^\infty \left( F^\otimes_+(x,x)\right)^2 |u(x,t)|^{p+3}\;dx
    \end{split}\end{equation*}
    Now, by~\Cref{thm:ext-mora-est}, the first two integrals in the last inequality are integrable in time, so 
    \begin{equation*}
        \int_1^T \eqref{eqn:EIM-nonlin} \;dt = 2\frac{p-1}{p+1} \int_1^T \int_{-\infty}^\infty \left( F^\otimes_+(x,x)\right)^2 |u(x,t)|^{p+3}\;dx dt + O\left( \sup_t \lVert u \rVert_{L^2}^3 \lVert u \rVert_{H^1}\right)
    \end{equation*}
    Finally, for~\eqref{eqn:EIM-linear}, we write
    \begin{subequations}\begin{align}
        \eqref{eqn:EIM-linear} =& \langle [-\Delta_{x,y}, F^\otimes_+] i\gamma^\otimes_+ F^\otimes_+ + F^\otimes_+ i\gamma^\otimes_+ [-\Delta_{x,y}, F^\otimes_+] \rangle_{t,U}\label{eqn:EIM-lin-cutoff}\\
        &+ \langle F^\otimes_+ i[-\Delta_{x,y}, \gamma^\otimes_+] F^\otimes_+ \rangle_{t,U}\label{eqn:EIM-lin-main}
    \end{align}\end{subequations}
    For the first term, we observe that
    \begin{equation*}\begin{split}
        [-\Delta_{x,y}, F^\otimes_+] =& -t^{-2\beta}(\Delta_{x,y} F^\otimes_+) - 2 t^{-\beta} (\partial_x F^\otimes_+) \partial_x - 2 t^{-\beta} (\partial_y F^\otimes_+) \partial_y\\
        =& t^{-2\beta}(\Delta_{x,y} F^\otimes_+) - 2 t^{-\beta} \partial_x (\partial_x F^\otimes_+) - 2 t^{-\beta} \partial_y (\partial_y F^\otimes_+)
    \end{split}\end{equation*}
    which lets us write
    \begin{subequations}\begin{align}
        \eqref{eqn:EIM-lin-cutoff} =& -t^{-2\beta} \left\langle (\Delta_{x,y} F^\otimes_+) i\gamma^\otimes_+ F^\otimes_+ - F^\otimes_+ i\gamma^\otimes_+ (\Delta_{x,y} F^\otimes_+) \right\rangle_{t,U} \label{eqn:EIM-lin-cutoff1}\\
                                    & -2t^{-\beta} \left\langle (\partial_x F^\otimes_+) \partial_x i\gamma^\otimes_+ F^\otimes_+ +   F^\otimes_+ i\gamma^\otimes_+ \partial_x (\partial_x F^\otimes_+)\right\rangle_{t,U} \label{eqn:EIM-lin-cutoff2}\\
                                    &+ \{\text{similar}\}\notag
    \end{align}
    \end{subequations}
    where $\{\text{similar}\}$ denotes the terms with $y$ derivatives.  For~\eqref{eqn:EIM-lin-cutoff1}, applying the symmetrization formula given in~\Cref{lem:two-term-sym} shows that
    \begin{equation*}
        \eqref{eqn:EIM-lin-cutoff1} = O(t^{-3\beta} \lVert u \rVert_{L^2}^4)
    \end{equation*}
    which is acceptable.  Turning to~\eqref{eqn:EIM-lin-cutoff2}, we again observe that for $x > y$, the vector field $\gamma^\otimes_+  \gamma_x$ and $\partial_x F^\otimes_+$ is supported in the region $|x| \sim t^\beta$, while for $x < y$, the vector field is $\gamma^\otimes_+ = \gamma_y$ and $\partial_x F^\otimes_+$ is supported in the region $|x|, |y| \sim t^\beta$.  Thus,
    \begin{equation*}
        |\eqref{eqn:EIM-lin-cutoff2}| \leq \left\lVert \gamma \sqrt{\tilde{F}'\tilde{F}(|x| \geq t^\beta)} u \right\rVert_{L^2}^2 \lVert u \rVert_{L^2}^2
    \end{equation*}
    so~\Cref{thm:ext-mora-est} implies that
    \begin{equation*}
        \int_1^T \eqref{eqn:EIM-lin-cutoff2} \;dt \lesssim \sup_t \lVert u \rVert_{L^2}^3 \lVert u \rVert_{H^1}
    \end{equation*}
    In particular, this shows that~\eqref{eqn:EIM-lin-cutoff} is perturbative.  Turning to~\eqref{eqn:EIM-lin-main}, we observe that by the same reasoning used in~\Cref{prop:gamma-otimes-ident},
    \begin{equation*}\begin{split}
        \eqref{eqn:EIM-lin-main} =& \int_{-\infty}^\infty |\partial_x (F^\otimes_+ U)|^2 + |\partial_x (F^\otimes_+ U)|^2 - F^\otimes_+ U \partial_x^2( F^\otimes_+ \overline{U}) - F^\otimes_+ U \partial_y^2( F^\otimes_+ \overline{U}) \\
        &\qquad\qquad + 2(F^\otimes_+ U \partial_x\partial_y (F^\otimes_+ \overline{U})) - 2 \Re(\partial_x (F^\otimes_+ U) \partial_y (F^\otimes_+ \overline{U}))\;dx
    \end{split}\end{equation*}
    Notice that each term in the integral contains three derivative.  Integrating by parts in the terms containing $\partial_x^2$ and $\partial_y^2$ and distributing all derivatives, there are three possibilities: either no derivatives hit the $F^\otimes_+$ cutoff functions, exactly one derivative hits the cutoffs, or both derivatives hit the cutoff.  Thus, we find that
    \begin{subequations}\begin{align}
        \Re \eqref{eqn:EIM-lin-main} =& \int_{-\infty}^\infty \left|\partial_x \left(F^\otimes_+(x,x) |u(x)|^2\right)\right|^2\;dx \label{eqn:EIM-lin-main1}\\
                                      &+ 4\int_{-\infty}^\infty (F^\otimes_{+}(x,x))^2 |\partial_x u|^2 |u|^2\;dx \label{eqn:EIM-lin-main2}\\
                                      &+ 2\int_{-\infty}^\infty (F^\otimes_{+}(x,x))^2 \left(\Re(\overline{u}^2 (\partial_x u)^2 -  |u|^2 |\partial_x u|^2\right)\;dx \label{eqn:EIM-lin-main3}\\
                                      &+ \int_{-\infty}^\infty \partial_x (F^\otimes_+(x,x))^2 O(|u|^3|\partial_x u|)\;dx \label{eqn:EIM-lin-main4}\\
                                      &+ \int_{-\infty}^\infty \partial_x^2 (F^\otimes_+(x,x))^2 O(|u|^4)\;dx\label{eqn:EIM-lin-main5}
    \end{align}\end{subequations}
    As in the proof of~\Cref{prop:gamma-otimes-ident}, $\eqref{eqn:EIM-lin-main2} + \eqref{eqn:EIM-lin-main3} \geq 0$.  For~\eqref{eqn:EIM-lin-main4}, we observe that by~\Cref{lem:GNS-variant},
    \begin{equation*}
        \lVert \tilde{F}'\tilde{F}(|x| \geq t^\beta) |u|^2 \rVert_{L^\infty} \lesssim \lVert \gamma \sqrt{\tilde{F}'\tilde{F}(|x| \geq t^\beta)} u \rVert_{L^2} \lVert u \rVert_{L^2}
    \end{equation*}
    Thus, using the fact that $\partial_x F^\otimes_+$ is supported in the region $|x| \sim t^\beta$ and has size $O(t^{-\beta})$, we see that
    \begin{equation*}\begin{split}
        |\eqref{eqn:EIM-lin-main4}| \lesssim& t^{-\beta} \left\lVert  \gamma \sqrt{\tilde{F}'\tilde{F}(|x| \geq t^\beta} u \right\rVert_{L^2} \rVert_{L^2} \lVert |u|^2 \rVert_{L^\infty}\\
        \lesssim& t^{-\beta} \left\lVert  \gamma \sqrt{\tilde{F}'\tilde{F}(|x| \geq t^\beta} u \right\rVert_{L^2}^2  \lVert u \rVert_{L^2}^2
    \end{split}\end{equation*}
    Similarly, we have that
    \begin{equation*}\begin{split}
        |\eqref{eqn:EIM-lin-main5}| \lesssim& t^{-2\beta} \lVert \partial_x^2(F^\otimes_+(x,x))^2 \rVert_{L^1} \lVert \sqrt{\tilde{F}'\tilde{F}(|x| \geq t^\beta)} |u|^2 \rVert_{L^\infty}^2\\
        \lesssim& t^{-\beta} \lVert \gamma \sqrt{\tilde{F}'\tilde{F}(|x| \geq t^\beta)} u \rVert_{L^2}^2 \lVert u \rVert_{L^2}^2
    \end{split}
    \end{equation*}
    so
    \begin{equation*}
        \int_1^T \eqref{eqn:EIM-lin-main4} + \eqref{eqn:EIM-lin-main5} \;dt \lesssim \sup_t \lVert u \rVert_{L^2}^3 \lVert u \rVert_{H^1}
    \end{equation*}
    
    Since $|\langle F^\otimes_+ \gamma^\otimes_+ F^\otimes_+ \rangle_{t,U}| \lesssim \sup_t \lVert u \rVert_{L^2}^3 \lVert u \rVert_{H^1}$, we conclude that for any $T \in (1,\infty)$
    \begin{equation*}
        \int_1^T \int_{-\infty}^\infty \left|\partial_x \left(F^\otimes_+(x,x) |u(x,t)|^2\right)\right|^2 + 2\frac{p-1}{p+1}(F^\otimes_+(x,x))^2 |u(x,t)|^{p+3}\;dxdt \lesssim \sup_t \lVert u \rVert_{L^2}^3 \lVert u \rVert_{H^1}
    \end{equation*}
    Taking $T \to \infty$ then gives the result.
\end{proof}}

In particular, by combining~\Cref{thm:ext-int-Mora} with~\Cref{lem:GNS-variant}, we can show that $u$ must decay in time in an averaged sense:
\begin{cor}\label{thm:L-infty-decay}
    For $\beta \in (1/3,1)$,
    \begin{equation}\label{eqn:L-infty-decay-eq}
        \lVert {\tilde{F}}^2(|x| \geq t^\beta) |u|^{{4}} \rVert_{L^q(1,\infty; L^\infty_x)} \lesssim \lVert u \rVert_{L^\infty_t H^1}^{{4}}
    \end{equation}
    whenever $q \geq \frac{{2}(p+1)}{p+3}$.
\end{cor}
\begin{proof}
    The estimate for $q = \infty$ is immediate by Sobolev embedding, so it suffices to prove~\eqref{eqn:L-infty-decay-eq} with $q = \frac{{2}(p+1)}{p+3}$. From~\Cref{lem:GNS-variant}, we have that
    \begin{equation*}\begin{split}
        \lVert {\tilde{F}}^2 |u|^{{4}} \rVert_{L^\infty_x} \lesssim& \lVert {\tilde{F}}\partial_x |u|^2 \rVert_{L^2_x} \lVert {\tilde{F}}|u|^2 \rVert_{L^2_x}%\\
        %\lesssim& \lVert F\partial_x |u|^2 \rVert_{L^2_x} \lVert F^{2/(p+3)}u \rVert_{L^{p+3}_x}^{\frac{p+3}{2(p+1)}} \lVert u \rVert_{L^2_x}^{\frac{p-1}{2(p+1)}}
    \end{split}\end{equation*}
    {Moreover, we see that}
    {\begin{equation*}
        \lVert {\tilde{F}} |u|^2 \rVert_{L^2} = \left\lVert {\tilde{F}}^{\frac{\editadd{p-1}}{p+1}} \left( F^{\frac{2}{p+3}} \editadd{|}u\editadd{|} \right)^{\frac{p+3}{p+1}} \editadd{|} u \editadd{|}^{\frac{p-1}{p+1}} \right\rVert_{L^2} \lesssim \editadd{\lVert {\tilde{F}}^{\frac{\editadd{p-1}}{p+1}} \rVert_{L^\infty}} \lVert {\tilde{F}}^{\frac{2}{p+3}} u \rVert_{L^{p+3}}^{\frac{p+3}{p+1}} \lVert u \rVert_{L^2}^{\frac{p-1}{p+1}}
    \end{equation*}}
    so the result follows by~\Cref{thm:ext-int-Mora}.
\end{proof}

%TODO: Complete review/edits below here

\section{Bounds for localized nonradiative states}\label{sec:loc-nonrad}
Now, we will prove~\Cref{thm:nonrad-spreading}.  Suppose that $|x|^{1/2}u_0 \in L^2$, and 
$$\lim_{t \to \infty} \langle F(|x| \geq t^\beta) \gamma F(|x| \geq t^\beta)\rangle_t = 0$$
Recall that by~\Cref{thm:ext-mora-est}, we have that
\begin{equation}\label{eqn:F-gamma-F-spreading}\begin{split}
    \langle F(|x| \geq t^\beta) \gamma F(|x| \geq t^\beta)\rangle_t \lesssim& \int_t^\infty \Re \left\langle \frac{|x|}{s^{\beta + 1}}F' \gamma F \right\rangle_s + s^{-3\beta}\langle G(x/s^\beta) \rangle_s\;ds\\
    &- \int_t^\infty 4s^{-\beta}\lVert \gamma\sqrt{F'F} u(s) \rVert_{L^2}^2 - \frac{s^{-\beta}}{p+1} \lVert \sqrt[p+1]{F'F} u(s) \rVert_{L^{p+1}}^{p+1}\;ds\\
    &+ O(t^{1-\beta(\sigma+1)} + \editadd{t^{\beta-1}})\lVert u \rVert_{L^2}^2
\end{split}\end{equation}
where $G(x/s^\beta)$ is a bounded function supported in the region $|x| \sim s^\beta$.
Now, 
\begin{equation*}
    F' \gamma F = \sqrt{F'F}\gamma \sqrt{F'F} + \sqrt{F'}\left([\gamma, \sqrt{F}]\sqrt{F'} - [\gamma, \sqrt{F'}]\sqrt{F}\right) \sqrt{F}
\end{equation*}
Since $[\gamma, f(x)] = -if'(x) \sgn(x)$ and $F, F' > 0$, we have that
\begin{equation}\label{eqn:F-gamma-F-spreading-absorb-1}\begin{split}
    \Re \left\langle \frac{|x|}{s^{\beta + 1}}F' \gamma F \right\rangle_s =& \Re \left\langle \frac{|x|}{s^{\beta + 1}}\sqrt{F'F} \gamma \sqrt{F'F} \right\rangle_s\\
    \leq& s^{-1} \lVert u(s) \rVert_{L^2} \lVert \gamma \sqrt{F'F} u(s)\rVert_{L^2}\\
     \leq& s^{-\beta} \lVert \gamma \sqrt{F'F} u(s)\rVert_{L^2}^2 + \frac{C}{s^{2-\beta}} \lVert u \rVert_{L^2}^2
\end{split}\end{equation}
For the second term, using the $L^2$ conservation, we see that
\begin{equation*}
    s^{-3\beta}\langle G(x/s^\beta) \rangle_s \lesssim s^{-3\beta} \lVert u \rVert_{L^2}^2
\end{equation*}
\editadd{Using the above inequalities to estimate the integrals in~\eqref{eqn:F-gamma-F-spreading}, we see that
\begin{equation}\label{eqn:F-gamma-F-est-all-p}
    \langle F(|x| \geq t^\beta) \gamma F(|x| \geq t^\beta) \rangle_t \leq O(t^{1-3\beta} + t^{1-3\beta} + t^{1-\beta(\sigma +1)}) \lVert u \rVert_{L^2}^2
\end{equation}
(Note that we have used the $-4\int_t^\infty s^{-\beta} \lVert \gamma \sqrt{F'F} u \rVert_{L^2}^2\;ds$ term in~\eqref{eqn:F-gamma-F-spreading} to absorb the corresponding term in~\eqref{eqn:F-gamma-F-spreading-absorb-1}).
}

\editadd{For the mass-subcritical problem, it is possible to get an improvement by controlling $\langle G(x/s^\beta) \rangle_t$ using the nonlinear term. We have that} $G(x/s^\beta) \lesssim \sqrt{\tilde{F}\tilde{F}'(x \geq  s^\beta)}$ \editadd{is supported in a region with volume $O(s^\beta)$, so H\"older's inequality gives}
\begin{equation*}\begin{split}
    s^{-3\beta}\langle G(x/s^\beta) \rangle_s \lesssim& \editadd{s^{-3\beta} s^{\left(1 - \frac{2}{p+1}\right)\beta}}\lVert \sqrt[p+1]{\tilde{F}\tilde{F}'(x \geq  s^\beta)} u \rVert_{L^{p+1}}^2\\
    \lesssim& s^{-2\beta} \left( s^{-\beta} \int \tilde{F}\tilde{F}'(x \geq  s^\beta) |u|^{p+1}\;dx\right)^{\frac{2}{p+1}}
\end{split}\end{equation*}
\editadd{Applying Young's inequality, we see that
\begin{equation*}
    |s^{-3\beta}\langle G(x/s^\beta) \rangle_s| \lesssim C s^{-2\frac{p+1}{p-1} \beta} + \frac{s^{-\beta}}{p+1} \int \tilde{F}\tilde{F}'(x \geq  s^\beta) |u|^{p+1}\;dx
\end{equation*}
In particular, arguing as above, we find that
\begin{equation*}
    \langle F(|x| \geq t^\beta) \gamma F(|x| \geq t^\beta) \rangle_t \leq Ct^{1 - 2 \frac{p+1}{p-1}\beta} + O(t^{1 - \beta(\sigma+1)} + t^{\beta - 1}) \lVert u \rVert_{L^2}^2
\end{equation*}
Since $-2 \frac{p+1}{p-1} < -3$ for $p < 5$, this gives an improvement in the mass-subcritical case.
}

Now, a quick calculation shows that
\begin{equation}\label{eqn:FxF-est}
    \partial_t \langle F(|x| \geq t^\beta)|x|F(|x| \geq t^\beta) \rangle_t = 2\langle F\gamma\editadd{F} \rangle_t + 2\langle \sqrt{FF' \frac{|x|}{t^\beta}} \gamma \sqrt{FF' \frac{|x|}{t^\beta}} \rangle_t + \langle FF' \frac{|x|^2}{t^{1+\beta}}\rangle_t
\end{equation}
\editadd{Integrating the above equality in time} and using~\Cref{lem:gamma-annulus-int} to control the second term, we see that
\begin{equation}
    \langle F|x|F \rangle_t \lesssim t^\beta\editadd{\lVert u \rVert_{L^2}^2} + \min(\editadd{t^{2 - 2 \frac{p+1}{p-1}\beta}}, t^{2-3\beta}\editadd{\lVert u \rVert_{L^2}^2}) + t^{2-(\sigma+1)\beta}\editadd{\lVert u \rVert_{L^2}^2}
\end{equation}
\editadd{Since $\lVert (1-F)|x| u \rVert_{L^2} \lesssim t^\beta \lVert u \rVert_{L^2}$, we conclude that}
\begin{equation}\label{eqn:x-growth-beta}
    \editadd{\langle |x| \rangle_t \lesssim t^\beta\editadd{\lVert u \rVert_{L^2}^2} + \min(\editadd{t^{2 - 2 \frac{p+1}{p-1}\beta}}, t^{2-3\beta}\editadd{\lVert u \rVert_{L^2}^2}) +  t^{2-(\sigma+1)\beta}\editadd{\lVert u \rVert_{L^2}^2})}
\end{equation}
\editadd{Now, our assumptions in~\Cref{thm:nonrad-spreading} only allow us to conclude~\eqref{eqn:x-growth-beta} for a single value of $\beta$.  In order to prove~\Cref{thm:nonrad-spreading}, we must show that we can obtain~\eqref{eqn:x-growth-beta} for all $\beta \in (1/3, 1)$.  The following lemma shows that this is possible:}
\editadd{\begin{lemma}\label{lem:all-beta-lem}
    Suppose that
    \begin{equation*}
        \lim_{t \to \infty} \frac{\langle |x| \rangle_t}{t} = 0
    \end{equation*}
    Then, for any $\beta \in (1/3, 1)$, 
    $$\Gamma := \lim_{t \to \infty} \langle F(|x| \geq t^\beta) \gamma F(|x| \geq t^\beta) = 0$$
\end{lemma}
\begin{proof}
    First, we show that the limit defining $\Gamma$ exists and is finite.  Using~\Cref{thm:ext-mora-est}, we see that
    \begin{equation*}\begin{split}
        \Gamma =& \lim_{T \to \infty} \langle F \gamma F \rangle_1 + \int_1^T 2s^{-\beta} \lVert \gamma \sqrt{F'F} u(s) \rVert_{L^2}^2 + {2}s^{-\beta}\frac{p-1}{p+1} \int F'F |u(s)|^{p+1}\;dx\;ds\\
        &+ O(\lVert u \rVert_{L^2}^2)\\
        =& \langle F \gamma F \rangle_1 + \int_1^\infty 2s^{-\beta} \lVert \gamma \sqrt{F'F} u(s) \rVert_{L^2}^2 + {2}s^{-\beta}\frac{p-1}{p+1} \int F'F |u(s)|^{p+1}\;dx\;ds\\
        &+ O(\lVert u \rVert_{L^2}^2)\\
        <& \infty\\
    \end{split}\end{equation*}

    We now show that $\Gamma = 0$.  From~\eqref{eqn:FxF-est} (which did not depend on the assumption $\Gamma = 0$), we have that
    \begin{equation*}
        \langle F(|x| \geq t^\beta) |x| F(|x| \geq t^\beta) \rangle_{t=T} = 2 \int_1^T \langle F \gamma F \rangle_t\;dt + O(t^\beta\lVert u \rVert_{L^2}^2)
    \end{equation*}
    The integral term is equal to  $2\Gamma T + o(T)$ as $T \to \infty$.  Since 
    $$0 \leq \langle F(|x| \geq t^\beta) |x| F(|x| \geq t^\beta) \rangle_{t} \leq \langle |x| \rangle_{t} = o(t)$$ this implies that $\Gamma = 0$.
\end{proof}}

\editadd{Thus, in~\eqref{eqn:x-growth-beta}, we are allowed to optimize over $\beta \in(1/3, 1)$.}  For $p \geq \editadd{5}$, \editadd{the optimal choice is} $\beta = \frac{1}{2}$, yielding
\begin{equation*}
    \langle |x| \rangle_t \lesssim t^{1/2}
\end{equation*}
On the other hand, if $\editadd{\frac{7}{3} < p < 5}$, then we can choose $\beta = \frac{2(p-1)}{3p+1} < \frac{1}{2}$ to obtain
$$\langle |x| \rangle_t \lesssim t^\beta$$
\editadd{while for $1 < p \leq \frac{7}{3}$, we get the above inequality for any $\beta \in (1/3, 1)$, since $2 - 2\frac{p-1}{p+1} \beta < \beta$ for $\beta > 1/3$.}  This concludes the proof of~\Cref{thm:nonrad-spreading}.

\section{Mass supercritical case: Construction of the free channel}
Now, we will show how to construct the free channel wave operator for~\eqref{eqn:NLS}:
\begin{equation}\label{eqn:free-channel-wave-op}
    \Omega_\text{free} u_0 := \slim_{t \to \infty} e^{-it\Delta} \Jfree(t) u(t)
\end{equation}
Here, the free channel restriction operator $\Jfree$ is defined to be
\begin{equation*}
    \Jfree(t) = e^{it\Delta}F(|x| \leq t^\alpha)e^{-it\Delta} F(|D| \geq t^{-\delta})
\end{equation*}
In this section, we will prove the following result:
\begin{thm}\label{thm:wave-op-exists}
    For positive $\alpha$ and $\delta$ satisfying
    \begin{equation}\label{eqn:alpha-delta-restrs}\begin{split}
        \alpha <& \alpha_0 := \frac{(p-5)(p+2)}{4(p+1)}\\
        \delta <& \min(1/2, \alpha)
    \end{split}\end{equation}
    the limit defining the free channel wave operator~\eqref{eqn:free-channel-wave-op} exists in the strong $H^1$ sense.
\end{thm}
Using Cook's method, we can write
\begin{subequations}\begin{align}
    \Omega_\text{free} u_0 -e^{-i\Delta}\Jfree(1)u(1) =
    &- i \slim\limits_{T \to \infty} \int_1^T \tilde{J}_\text{free}e^{-it\Delta}  Vu(t)\;dt\label{eqn:pot-err}\\
    &-i \slim\limits_{T \to \infty} \int_1^T \tilde{J}_\text{free}e^{-it\Delta}  |u|^{p-1}u(t) \; dt\label{eqn:nonlin-err}\\
    &+ \slim\limits_{T \to \infty} \int_1^T \frac{d}{dt} \tilde{J}_\text{free}(t)e^{-it\Delta}  u(t) \;dt \label{eqn:J-err}
\end{align}\end{subequations}
where
\begin{equation*}
    \tilde{J}_\text{free} = F(|x| \leq t^\alpha) F(|D| \geq t^{-\delta})
\end{equation*}
To prove~\Cref{thm:wave-op-exists}, we must show that each of the integrals~\eqref{eqn:pot-err}--\eqref{eqn:J-err} converges as $T \to \infty$.  We will show how to get convergence in $L^2$, and then sketch the modifications necessary to extend the convergence to $H^1$.

\subsection{The potential term}\label{sec:potential}
We begin by considering the potential term~\eqref{eqn:pot-err}.  Let us write
\begin{equation*}
    \eqref{eqn:pot-err} = -i\slim_{T \to \infty} \int_1^T \tilde{J}_\text{free}e^{-it\Delta}  F\left(|x| \geq t^\beta\right)Vu(t)\;dt\\
    -i\slim_{T \to \infty} \int_1^T \tilde{J}_\text{free}e^{-it\Delta}  F\left(|x| \leq t^\beta\right)Vu(t)\;dt
\end{equation*}
where
\begin{equation}\label{eqn:beta-def}
    \max(\frac{1}{3}, \frac{1}{\sigma}) < \beta < 1 - \delta
\end{equation}
For the first term, the definition of $\beta$ guarantees that $\lVert F(|x| \geq t^\beta)V \rVert_{L^\infty_x}$ is integrable in time, which allows us to control the integral.  On the other hand, since $\delta < \min(1-\alpha, 1/2)$, we can use~\Cref{lem:st-ph-lemma-high-freq} to show that the integrand of the second term also decays at an integrable rate.

\subsection{The nonlinear term}\label{sec:nonlin}
We now turn to the nonlinear term~\eqref{eqn:nonlin-err}.  As before, we decompose the interaction term into a piece close to the origin and a piece far from the origin:
\begin{equation*}\begin{split}
    e^{-it\Delta} \Jfree  |u|^{p-1}u =& \tilde{J}_\text{free} e^{-it\Delta} F(|x| \leq t^\beta) |u|^{p-1}u\\
    &+ \tilde{J}_\text{free} e^{-it\Delta}F(|x| \geq t^\beta) |u|^{p-1}u
\end{split}\end{equation*}
For the term localized near the origin, we can again use~\Cref{lem:st-ph-lemma-high-freq} combined with the Sobolev embedding to conclude that
\begin{equation*}
    \lVert \tilde{J}_\text{free} e^{-it\Delta} F(|x| \leq t^\beta) |u|^{p-1}u \rVert_{L^2} \lesssim t^{-2} \lVert u \rVert_{H^1}^{p}
\end{equation*}
which is integrable in time.  In the exterior region $|x| \geq t^{\beta}$, we have that
\begin{equation*}\begin{split}
    \lVert \tilde{J}_\text{free} e^{-it\Delta}  F(|x| \geq t^\beta) |u|^{p-1}u \rVert_{L^2} \lesssim& t^{-\frac{1}{2}(1 - \alpha)} \lVert F(|x| \geq t^\beta) |u|^{p-1}u \rVert_{L^1_x}\\
    \lesssim& t^{-\frac{1}{2}(1 - \alpha)} \lVert u \rVert_{L^2}^2 \lVert \tilde{F}(|x| \geq t^\beta) |u|^{2} \rVert_{L^\infty}^{\frac{p-2}{2}}
\end{split}\end{equation*}
where $\tilde{F}$ is a cut-off function such that $\tilde{F} = 1$ on the support of $F$.

From~\Cref{thm:L-infty-decay}, we have that $\tilde{F} |u|^2 \in L^{\frac{4(p+1)}{p+3}}_tL^\infty_x$.  Thus, the term in the exterior region will be integrable in time provided that
\begin{equation*}
    \frac{(p+3)(p-2)}{8(p+1)} + \frac{1-\alpha}{2} > 1
\end{equation*}
which simplifies to
\begin{equation*}
    \alpha < \frac{(p-5)(p+2)}{4(p+1)} = \alpha_0
\end{equation*}

\subsection{The channel restriction error}\label{sec:channel-restr}
Finally, we consider the term~\eqref{eqn:J-err} containing $\frac{d}{dt} \tilde{J}_\text{free}$.  We can write
\begin{align*}
    \eqref{eqn:J-err} =& \slim\limits_{T \to \infty} \int_1^T C_1(t) e^{-it\Delta} u(t) \;dt + \slim\limits_{T \to \infty} \int_1^T C_2(t) e^{-it\Delta} u(t) \;dt
\end{align*}
with
\begin{equation*}\begin{split}
    C_1(t) =& \partial_t F(|x| \leq t^\alpha) F(|D| \geq t^{-\delta})\\
    C_2(t) =& F(|x| \leq t^\alpha) \partial_t F(|D| \geq t^{-\delta})
\end{split}
\end{equation*}
To start, we will show that the improper integral containing $C_1(t)$ exists by showing that it is Cauchy.  Let $1 \leq t_0 < t_1 < \infty$.  Then, by duality, we can write
\begin{equation}\label{eqn:crosstalk-1-Cauchy}\begin{split}
    \biggl\lVert \int_{t_0}^{t_1} C_1(t) e^{-it\Delta} u(t) \;dt \biggr\rVert_{L^2} =& \sup_{\lVert \varphi \rVert_{L^2} = 1} \int_{t_0}^{t_1} \left\langle \varphi,  C_1(t) e^{-it\Delta} u(t) \right\rangle \;dt\\
    =& \sup_{\lVert \varphi \rVert_{L^2} = 1} \int_{t_0}^{t_1} \biggl\langle \sqrt{\partial_t F(|x| \leq t^\alpha)} \varphi,  \sqrt{\partial_t F(|x| \leq t^\alpha)} F(|D| \geq t^{-\delta}) e^{-it\Delta} u(t) \biggr\rangle \;dt
\end{split}\end{equation}
Since $\varphi$ is time-independent, we have that
\begin{equation}\label{eqn:psi-est}\begin{split}
    \int_{t_0}^{t_1} \left\lVert \sqrt{\partial_t F(|x| \leq t^\alpha)} \varphi \right\rVert_{L^2}^2\;dt =& \left\langle \varphi, \int_{t_0}^{t_1} \partial_t F(|x| \leq t^\alpha)\;dt \varphi \right\rangle\\
    =& \left\langle \varphi, \left(F(|x| \leq t_1^\alpha) - F(|x| \leq t_0^\alpha)\right) \varphi \right\rangle\\
    \leq& \lVert \varphi \rVert_{L^2}^2
\end{split}\end{equation}
On the other hand, by considering the evolution of the propagation observable
\begin{equation*}
    B(t) =  F(|D| \geq t^{-\delta}) F(|x| \leq t^\alpha) F(|D| \geq t^{-\delta})
\end{equation*}
when applied to $\phi(t) = e^{-it\Delta} u(t)$, we compute that
\begin{subequations}\begin{align}
    \partial_t \langle B(t) \rangle_{t,\phi} =& \langle \partial_t B(t) \rangle_{t,\phi} + \langle i[e^{-it\Delta} (V(t) + |u|^{p-1}) e^{it\Delta}, B(t)] \rangle_{t,\phi}\notag\\
    =& \langle F(|D| \geq t^{-\delta}) \partial_t F(|x| \leq t^\alpha) F(|D| \geq t^{-\delta}) \rangle_{t,\phi}\label{eqn:crosstalk-1-a}\\
    &+ \langle \partial_t F(|D| \geq t^{-\delta}) F(|x| \leq t^\alpha) F(|D| \geq t^{-\delta}) \rangle_{t,\phi}\label{eqn:crosstalk-1-b}\\
    &+ \langle F(|D| \geq t^{-\delta}) F(|x| \leq t^\alpha) \partial_t F(|D| \geq t^{-\delta}) \rangle_{t,\phi}\label{eqn:crosstalk-1-c}\\
    &+ \langle i[e^{-it\Delta} V(t) e^{it\Delta}, B(t)] \rangle_{t,\phi}\label{eqn:crosstalk-1-d}\\
    &+ \langle i[e^{-it\Delta} |u|^{p-1} e^{it\Delta}, B(t)] \rangle_{t,\phi}\label{eqn:crosstalk-1-e}
\end{align}\end{subequations}
\editadd{(Notice that this differs from~\eqref{eqn:obs-deriv}, since $\phi$ is not a solution of~\eqref{eqn:NLS}.)}  For~\eqref{eqn:crosstalk-1-d}, we note that
\begin{equation}\label{eqn:crosstalk-1-d-exp}\begin{split}
    \langle e^{-it\Delta} V(t) e^{it\Delta} B(t) \rangle_{t,\phi} =& \langle e^{-it\Delta} V(t) e^{it\Delta} F(|D| \geq t^{-\delta})F(|x| \leq t^\alpha) F(|D| \geq t^{-\delta}) \phi(t), \phi(t)\rangle_{L^2_x}\\
    =& \langle F(|D| \geq t^{-\delta}) \phi(t), F(|x| \leq t^\alpha)  e^{-it\Delta} F(|D| \geq t^{-\delta}) V u(t)\rangle_{L^2_x}
\end{split}
\end{equation}
From~\Cref{sec:potential}, we know that $\lVert F(|x| \leq t^\alpha)  e^{-it\Delta} F(|D| \geq t^{-\delta}) V u(t) \rVert_{L^2}$ is integrable in time, which shows that~\eqref{eqn:crosstalk-1-d-exp} is in $L^1_t(1,\infty)$.  A similar argument also applies to $\langle B(t) e^{-it\Delta} V(t) e^{it\Delta} \rangle_{t,\phi}$, which shows that~\eqref{eqn:crosstalk-1-d} is time integrable.  Similarly, for~\eqref{eqn:crosstalk-1-e}, we write
\begin{equation*}\begin{split}
    \langle e^{-it\Delta} |u|^{p-1} e^{it\Delta} B(t) \rangle_{t,\phi} =& \langle e^{-it\Delta} |u|^{p-1} e^{it\Delta}F(|D| \geq t^{-\delta}) F(|x| \leq t^\alpha) F(|D| \geq t^{-\delta}) \phi, \phi \rangle_{L^2_x}\\
    =& \langle  F(|D| \geq t^{-\delta}) \phi, F(|x| \leq t^\alpha) e^{-it\Delta} F(|D| \geq t^{-\delta})|u|^{p-1}u \rangle_{L^2_x}\\
\end{split}\end{equation*}
and note that $F(|x| \leq t^\alpha) e^{-it\Delta} F(|D| \geq t^{-\delta})|u|^{p-1}u$ was shown to be in $L^1_tL^2_x$ in section~\Cref{sec:nonlin}, showing that~\eqref{eqn:crosstalk-1-e} is integrable.  For~\eqref{eqn:crosstalk-1-b} and~\eqref{eqn:crosstalk-1-c}, we observe that
\begin{equation*}\begin{split}
    \left\lVert \left[F(|D| \geq t^{-\delta}), F(|x| \leq t^\alpha)\right] \right\rVert_{L^2 \to L^2} \lesssim& t^{-\alpha + \delta}\\
    \left\lVert \left[\partial_t F(|D| \geq t^{-\delta}), F(|x| \leq t^\alpha)\right] \right\rVert_{L^2 \to L^2} \lesssim& t^{-1-\alpha + \delta}\\
    \left\lVert \partial_t F(|D| \geq t^{-\delta}) \right\rVert_{L^2 \to L^2} \lesssim t^{-1}
\end{split}\end{equation*}
Thus, after symmetrizing using~\Cref{lem:three-term-sym}, we find that
\begin{equation*}
    \eqref{eqn:crosstalk-1-b}+\eqref{eqn:crosstalk-1-c} = 2\left\lVert \sqrt{F(|x| \leq t^\alpha) F(|D| \geq t^{-\delta}) \partial_t F(|D| \geq t^{-\delta})} e^{-it\Delta} u(t) \right\rVert_{L^2}^2 + O(t^{-1-\alpha + \delta} \lVert u(t) \rVert_{L^2}^2)
\end{equation*}
is the sum of a positive term and an integrable term.  Since~\eqref{eqn:crosstalk-1-a} is positive, combining these estimates shows that
\begin{equation}\label{eqn:crosstalk-1-prop-est}
    \int_1^\infty \lVert \sqrt{\partial_t F(|x| \leq t^\alpha)} F(|D| \geq t^{-\delta}) e^{-it\Delta} u(t) \rVert_{L^2}^2\;dt \lesssim \lVert u_0 \rVert_{L^2}^2
\end{equation}
Thus, using Cauchy-Schwarz in~\eqref{eqn:crosstalk-1-Cauchy} and using~\eqref{eqn:psi-est} and~\eqref{eqn:crosstalk-1-prop-est}, we see that the integral containing $C_1(t)$ is strongly Cauchy in time.

Since the argument for the integral containing $C_2(t)$ is similar, we will sketch the necessary changes without going into detail.  Using commutator estimates, we see that
\begin{equation*}
    C_2(t) = \partial_t F(|D| \geq t^{-\delta})
  F(|x| \leq t^\alpha) + O_{L^2 \to L^2}(t^{-1-\alpha + \delta})
\end{equation*}
The leading order term is essentially the same as $C_1(t)$, but with the roles of $F(|x| \leq t^\alpha)$ and $F(|D| \geq t^{-\delta})$ swapped.  A cursory examination shows that the previous argument goes through with only minor changes under this switch, so the integral containing $C_2(t)$ also converges strongly in $L^2$.

\subsection{On the wave operator in \texorpdfstring{$H^1$}{H1}}
The above argument shows that the wave operator $\Omega_\text{free}$ is well defined in $L^2$ whenever $u(t)$ is bounded in $H^1$.  Examining the proof, the $H^1$ bound is only ever used to obtain the exterior $L^\infty$ bound from~\Cref{thm:L-infty-decay}.  As a result, we can show that $\partial_x \Omega_\text{free} u_0 \in L^2$, since there are no problems with loss of regularity.  We will sketch the changes necessary for this improvement.  Applying Cook's method, we obtain terms analogous to~\cref{eqn:pot-err,eqn:nonlin-err,eqn:J-err}:
\begin{subequations}\begin{align}
    \partial_x \Omega_\text{free} u_0
    =& e^{-i\Delta} \Jfree(1) u(1)\notag \\
    &- i \slim\limits_{T \to \infty} \partial_x\int_1^T \tilde{J}_\text{free}e^{-it\Delta}  Vu(t)\;dt\label{eqn:pot-err-deriv}\\
    &-i \slim\limits_{T \to \infty} \partial_x\int_1^T \tilde{J}_\text{free}e^{-it\Delta}  |u|^{p-1}u(t) \; dt\label{eqn:nonlin-err-deriv}\\
    &+ \slim\limits_{T \to \infty} \partial_x\int_1^T \frac{d}{dt} \tilde{J}_\text{free}(t)e^{-it\Delta}  u(t) \;dt \label{eqn:J-err-deriv}
\end{align}\end{subequations}
It only remains to prove that these limits exist in $L^2$.

\subsubsection{The potential term}
Examining the integrand of~\eqref{eqn:pot-err-deriv}, we find that
\begin{equation*}\begin{split}
    \partial_x (\tilde{J}_\text{free}e^{-it\Delta}  Vu(t)) =& [\partial_x, \tilde{J}_\text{free}] e^{-it\Delta} Vu(t)\\
    &+ \tilde{J}_\text{free}e^{-it\Delta} V' u(t)\\
    &+ \tilde{J}_\text{free}e^{-it\Delta} V \partial_x u(t)
\end{split}\end{equation*}
For the first term, we observe that the commutator $[\partial_x, \tilde{J}_\text{free}]$ has the same support properties as $\tilde{J}_\text{free}$ and is smaller by a factor of $t^{-\alpha}$, so the argument from~\Cref{sec:potential} goes through with no changes.  Since $V'$ has even better decay than $V$ in the exterior region, the same is true for the second term.  Finally, for the last term, we simply use the assumption of bounded energy instead of mass conservation when applying~\Cref{lem:st-ph-lemma-high-freq}.

\subsubsection{The nonlinear term}
Here, we have
\begin{equation*}\begin{split}
    \partial_x (\tilde{J}_\text{free}e^{-it\Delta}  |u|^{p-1}u(t)) =& [\partial_x, \tilde{J}_\text{free}] e^{-it\Delta} |u|^{p-1}u(t)(t)\\
    &+ \tilde{J}_\text{free}e^{-it\Delta} \left(\frac{p+1}{2} |u|^{p-1}\partial_x u + \frac{p-1}{2}|u|^{p-3} u^2 \partial_x \overline{u}\right)(t)\\
\end{split}\end{equation*}
The commutator term can be handled using the same modifications discussed above, while for the second term, the argument from~\Cref{sec:nonlin} goes through provided we place the $\partial_x u$ or $\partial_x \overline{u}$ factor in $L^2$.

\subsubsection{The channel restriction term}
Finally, we consider the channel restriction error term.  We have
\begin{equation*}\begin{split}
    \eqref{eqn:J-err-deriv} =& [\partial_x, \frac{d}{dt} \tilde{J}_\text{free}(t)] e^{-it\Delta} u(t)\\
    &+ F(|x| \leq t^\alpha) \partial_t F(|D| \geq t^{-\delta}) \partial_x e^{-it\Delta}  u(t)\\
    &+ \partial_t F(|x| \leq t^\alpha) F(|D| \geq t^{-\delta})e^{-it\Delta} \partial_x u(t)\\
\end{split}\end{equation*}
For the first term, we observe that $\lVert [\partial_x, \frac{d}{dt} \tilde{J}_\text{free}(t)] \rVert_{L^2 \to L^2} \lesssim t^{-1-\alpha}$, which is sufficient to place this term in $L^1_t$.  For the second term, $\partial_t F(|\xi| \geq t^{-\delta})$ is supported in the region $|\xi| \sim t^{-\delta}$, so
\begin{equation*}
    \lVert \partial_t F(|D| \geq t^{-\delta}) \partial_x \rVert_{L^2 \to L^2} \lesssim t^{-1-\delta}
\end{equation*}
which is integrable.  Finally, for the third term, we can argue by duality to reduce the problem to showing that 
\begin{equation*}
    \lVert \sqrt{\partial_t F(|x| \leq t^\alpha\editadd{)}} F(|D| \geq t^{-\delta}) \partial_x e^{-it\Delta} u(t) \rVert_{L^2_x} \in L^2_t
\end{equation*}
To show this, we perform a propagation estimate with respect to the observable
\begin{equation*}
    \partial_x B(t)\partial_x = \partial_x F(|D| \geq t^{-\delta}) F(|x| \leq t^\alpha)  F(|D| \geq t^{-\delta}) \partial_x
\end{equation*}
Since $\partial_x u$ is bounded in $L^2$ by assumption, $\langle \partial_x B(t) \partial_x \rangle_{t,\phi}$ is bounded in time.  Taking the Heisenberg derivative, we see that the argument for the terms~\cref{eqn:crosstalk-1-a,eqn:crosstalk-1-b,eqn:crosstalk-1-c} can be easily adapted by placing each of the derivatives on $\phi$.  (Note that the non-perturbative terms will now be negative definite instead of positive definite due to the integration by parts.) For~\eqref{eqn:crosstalk-1-d}, we see that
\begin{equation*}
    [\partial_x B(t) \partial_x, e^{-it\Delta}Ve^{-it\Delta}] = [\partial_x, e^{-it\Delta}Ve^{-it\Delta}]B(t) \partial_x + \partial_x B(t) [\partial_x, e^{-it\Delta}Ve^{-it\Delta}] + \partial_x [B(t), e^{-it\Delta}Ve^{-it\Delta}] \partial_x
\end{equation*}
The last term requires no modification once we place the derivatives on the $\phi$'s, while the first two terms can also be handled similarly using the decay hypothesis on $\partial_x V$.

The term~\eqref{eqn:crosstalk-1-e} requires slightly more modification.  As above, we have
\begin{equation*}\begin{split}
    [\partial_x B(t) \partial_x, e^{-it\Delta}|u|^{p-1}e^{-it\Delta}] =& [\partial_x, e^{-it\Delta}|u|^{p-1}e^{-it\Delta}]B(t) \partial_x + \partial_x B(t) [\partial_x, e^{-it\Delta}|u|^{p-1}e^{-it\Delta}]\\
    &+ \partial_x [B(t), e^{-it\Delta}|u|^{p-1}e^{-it\Delta}] \partial_x
\end{split}\end{equation*}
We will focus on the first term, since the modifications for the other terms are either similar or easier.  Here, we see that
\begin{equation*}
    [\partial_x, -ie^{-it\Delta}|u|^{p-1}e^{-it\Delta}] = -ie^{-it\Delta}(\partial_x |u|^{p-1})e^{-it\Delta}
\end{equation*}
Thus, we calculate that
\begin{equation*}\begin{split}
    |\langle [\partial_x, e^{-it\Delta}|u|^{p-1}e^{it\Delta}]B(t) \partial_x \rangle_{t,\phi}|=& |\langle e^{-it\Delta}(\partial_x |u|^{p-1})e^{-it\Delta}]B(t) \partial_x \phi, \phi \rangle_{L^2_x}|\\
    =& |\langle F(|D| \geq t^{-\delta}) \partial_x u(t), F(|x| \leq t^\alpha) e^{-it\Delta} F(|D| \geq t^{-\delta}) (\partial_x |u|^{p-1})u \rangle_{L^2_x}|
\end{split}
\end{equation*}
by the arguments used to control the nonlinear term,
\begin{equation*}
    \lVert F(|x| \leq t^\alpha) e^{-it\Delta} F(|D| \geq t^{-\delta}) (\partial_x |u|^{p-1})u \rVert_{L^1_tL^2_x} < \infty
\end{equation*}
which completes the proof that the channel restriction error is integrable in $H^1$.

\section{Mass supercritical case: Behavior of the remainder}
We now know that the part of the solution microlocalized to the support of $\Jfree$ behaves asymptotically like a free wave.  In this section, we will show that the mass and energy of the rest of the solution can spread only slowly.  In particular, we will show that:
\begin{thm}\label{thm:loc-state-spread}
    Let $\alpha$ and $\delta$ be the parameters introduced in~\Cref{thm:wave-op-exists}.  Then, for any $\kappa$ satisfying
    \begin{equation}\label{eqn:kappa-restr}
        \kappa > \max(1-\delta, \alpha, \beta)
    \end{equation}
    we have
    \begin{equation}\label{eqn:u-nonfree-bd-L2}
        \lim_{t \to \infty} \lVert F(|x| \geq t^\kappa) (I - \Jfree) u(t) \rVert_{L^2} = 0
    \end{equation}
    Similarly, for $\mu$ satisfying
    \begin{equation}\label{eqn:mu-restr}
        \mu > \beta
    \end{equation}
    we have improved localization of energy given by
    \begin{equation}\label{eqn:u-nonfree-bd-dot-H1}
        \lim_{t \to \infty} \lVert F(|x| \geq t^\mu) (I - \Jfree) \partial_x u(t) \rVert_{L^2} = 0
    \end{equation}
\end{thm}
\begin{rmk}
    Note that~\eqref{eqn:kappa-restr} and~\eqref{eqn:mu-restr} together with the bounds for the other parameters imply that $\kappa > 1/2$ and $\mu > \max(1/3, 1/\sigma)$.
\end{rmk}
Fix $\kappa$ and $\mu$ that satisfy~\eqref{eqn:kappa-restr} and~\eqref{eqn:mu-restr}.  We introduce the decomposition $I - \Jfree = J_\text{low} + J_\text{rem}$, where
\begin{equation}\begin{split}
    J_\text{low} =&  {e^{it\Delta}}F(|{x}| \leq t^\alpha) {e^{-it\Delta}} F(|D| \leq t^{-\delta})\\
    J_\text{rem} =& {e^{it\Delta}}F(|{x}| \geq t^\alpha) {e^{-it\Delta}}
\end{split}\end{equation}
and write $u_\text{low}(t) = J_\text{low}(t)u(t)$, $u_\text{rem}\editadd{(t)} = J_\text{rem}(t) u(t)$.  Observe that the energy of $u_\text{low}(t)$ goes to $0$ as $t \to \infty$.  Moreover, using~\Cref{lem:st-ph-lemma-low-freq} and noting the restriction~\eqref{eqn:kappa-restr} on $\kappa$, we see that the mass of $u_\text{low}(t)$ is concentrated in $|x| \leq t^\kappa$.
Thus, we have proved that
\begin{equation}\label{eqn:u-low-bounds}
    \lVert u_\text{low}(t) \rVert_{\dot{H}^1} + \lVert F(|x| \geq t^{\kappa}) u_\text{low}(t) \rVert_{L^2}  = o(1)
\end{equation}
To complete the proof of~\Cref{thm:loc-state-spread}, it suffices to show that
\begin{equation}\label{eqn:u-rem-L2-bound}
    \lVert F(|x| \geq t^{\kappa}) u_\text{rem}(t) \rVert_{L^2}  = o(1)
\end{equation}
and
\begin{equation}\label{eqn:u-rem-dot-H1-bound}
    \lVert F(|x| \geq t^{\mu}) \partial_x u_\text{rem}(t) \rVert_{L^2} = {o(1)}
\end{equation}
We will first show how to obtain the $L^2$ bound, and then explain the argument for the $\dot{H}^1$ bound.

\subsection{Slow spreading for \texorpdfstring{$u_\text{rem}$}{u\_rem}}
\subsubsection{Preliminary decomposition}
Using the Duhamel formula, we see that
\begin{equation}\label{eqn:u-rem-0-Duhamel}
    u_\text{rem}(t) = e^{it\Delta} F(|x| \geq t^\alpha) u_0 -i \int_0^t e^{it\Delta} F(|x| \geq t^\alpha) e^{-is\Delta} G(u(s))\;ds
\end{equation}
where for notational convenience we define $G(u) = Vu + |u|^{p-1}u$.  The term $e^{it\Delta} F(|x| \geq t^\alpha)u_0$ converges strongly to $0$ in $H^1$, so it can be neglected.  A similar argument shows that we can ignore the contribution to the Duhamel integral from $0 < s < T$ for any fixed parameter $T$.  In addition, a quick calculation using the decay hypotheses for $V$ and~\Cref{thm:L-infty-decay}
shows that
\begin{equation*}
    F(|x| \geq t^\beta) (Vu + |u|^{p-1}u)  \in L^{4/3}_t W^{1,1}_x
\end{equation*}
so, by using the dual Strichartz estimates from~\Cref{thm:dual-Strichartz}
we find that
\begin{equation*}
    \lim_{T \to \infty} \sup_{t > T} \left\lVert \int_T^t e^{-is\Delta} {G}(u({s}))\;ds \right\rVert_{H^1} = 0
\end{equation*}
Thus, we can write
\begin{equation}\label{eqn:u-rem-0-Duhamel-processed}
    u_\text{rem}(t) = -i \int_T^t e^{it\Delta}  F(|x| \geq t^\alpha) e^{-is\Delta} F(|x| \leq s^\beta) {G}(u(s)) \;ds + R_{T}(t)
\end{equation}
where $\lim_{T \to \infty} \sup_{t > T} \lVert R_T(t)\rVert_{H^1} = 0$.  On the other hand, for any $\psi \in H^1$, 
\begin{equation*}
    \langle \psi, e^{-it\Delta}J_\text{rem}(t) u(t) \rangle = \langle e^{it\Delta} F(|x| \geq t^\alpha) \psi, u(t) \rangle \to 0
\end{equation*}
so $\wlim_{t \to \infty} e^{-it\Delta} u_\text{rem}(t) = 0$.  It follows that we can write
\begin{equation}\label{eqn:u-rem-infty-Duhamel}\begin{split}
    u_\text{rem}(t) =& \int_t^\infty e^{it\Delta} \partial_s\left( e^{-is\Delta} J_\text{rem}(s) u(s)\right)\;ds\\
    =& -\alpha \int_t^\infty e^{it\Delta} \frac{x}{s^{\alpha + 1}}F'(|x| \geq s^\alpha) e^{-is\Delta} u(s)\;ds\\
    &+i \int_t^\infty e^{it\Delta} F(|x| \geq s^\alpha) e^{-is\Delta} (V + |u|^{p-1}) u(s)\;ds
\end{split} \end{equation}
where convergence for the integrals is understood in the weak sense.  A straightforward adaptation of the arguments from~\Cref{sec:channel-restr} shows that 
 \begin{equation*}
     \left\lVert \int_t^\infty \frac{x}{s^{\alpha + 1}}F'(|x| \geq s^\alpha) e^{-is\Delta} u(s)\;ds \right\rVert_{L^\infty_t H^1_x} \lesssim \lVert u \rVert_{H^1}\left(1 + \lVert u \rVert_{H^1}^{p-1}\right)
 \end{equation*}
 Thus,~\eqref{eqn:u-rem-infty-Duhamel} can be rewritten as
 \begin{equation}\label{eqn:u-rem-infty-Duhamel-processed}\begin{split}
    u_\text{rem}(t) =& i\int_t^\infty e^{it\Delta} F(|x| \geq s^\alpha) e^{-is\Delta}  F(|x| \leq s^\beta)G(u{(s)})\;ds + \tilde{R}(t)
\end{split}\end{equation}
where $\tilde{R}(t) = o_{H^1}(1)$.  The advantage of the representations~\eqref{eqn:u-rem-0-Duhamel-processed} and~\eqref{eqn:u-rem-infty-Duhamel-processed} is that the remainders $R_T(t)$ and $\tilde{R}(t)$ are negligible in~\eqref{eqn:u-rem-L2-bound} and~\eqref{eqn:u-rem-dot-H1-bound}.  Thus, in the discussion that follows we will drop the remainder terms in~\eqref{eqn:u-rem-0-Duhamel-processed} and~\eqref{eqn:u-rem-infty-Duhamel-processed}.

\subsubsection{An incoming/outgoing decomposition}
To prove localization for $u_\text{rem}$, we will employ an incoming/outgoing decomposition.  Let us define
\begin{equation}\label{eqn:in-out-proj}\begin{split}
    P^\text{out} =&  \bbOne_{x \geq 0} F(D \geq t^{-\lambda}) + \bbOne_{x < 0} F(-D \geq t^{-\lambda})\\
    P^\text{in} =&  \bbOne_{x \geq 0} F(-D \geq t^{-\lambda}) + \bbOne_{x < 0} F(+D \geq t^{-\lambda})\\
    P^\text{low} =& F(|D| \leq t^{-\lambda})
\end{split}\end{equation}
where
\begin{equation}\label{eqn:lambda-restr}
    1-\kappa < \lambda < \min\left(\frac{1}{2}, 1-\alpha, 1-\beta\right) 
\end{equation}
\begin{rmk}
    It is common in the literature to consider the alternative definitions
    \begin{align*}
        {P}^\textup{out} =& F(A) \\
        {P}^\textup{in} =& 1 - F(A)
    \end{align*}
    where $A = \frac{1}{2}(x \cdot D + D \cdot x)$ is the dilation operator and $F(x)$ is (a possibly smoothed version of) a cut-off to $x > 0$ \cite{mourre1979link}.  At least formally, our incoming and outgoing operators can be written as
    \begin{align*}
        P^\textup{out} =& \bbOne_{A > 0} F(|D| \geq t^{-\lambda})\\
        P^\textup{in} =& \bbOne_{A < 0} F(|D| \geq t^{-\lambda})
    \end{align*}
    Since we work in dimension one, the definition~\eqref{eqn:in-out-proj} allows us to avoid the technicalities associated with the operator calculus machinery.  It is also essentially the same as the original definition given by Enss~\cite{enssAsymptoticCompletenessQuantum1978}.
\end{rmk}
By combining this decomposition with~\eqref{eqn:u-rem-0-Duhamel-processed} and~\eqref{eqn:u-rem-infty-Duhamel-processed}, we will establish the bounds~\eqref{eqn:u-rem-L2-bound}.

\paragraph{\textit{Bounds for low frequencies}} We begin by considering the $L^2$ bound for $P^\text{low}\urem$.  Using the representation~\eqref{eqn:u-rem-0-Duhamel-processed}, we see that
\begin{equation*}\begin{split}
    F(|x| \geq t^\kappa) P^\text{low} \urem =& -i \int_T^t F(|x| \geq t^\kappa) F(|D| \leq t^{-\lambda}) e^{it\Delta}  F(|x| \geq t^\alpha) e^{-is\Delta} F(|x| \leq s^\beta) G(u(s)) \;ds\\
    &+ \{\textup{negligible terms}\}
\end{split}\end{equation*}
Note that we can write
\begin{equation*}
    e^{it\Delta}  F(|x| \geq t^\alpha) e^{-is\Delta} = e^{i(t-s)\Delta} - e^{it\Delta} F(|x| \leq t^\alpha) e^{-is\Delta}
\end{equation*}
so it suffices to establish the bounds
\begin{equation}\label{eqn:rem-low-op-bds}\begin{split}
    \int_T^t \lVert F(|x| \geq t^\kappa) F(|D| \leq t^{-\lambda}) e^{i(t-s)\Delta} F(|x| \leq s^\beta) \rVert_{L^2 \to L^2}\;ds := \rmI^\text{low}_1 =& o(1)\\
    \int_T^t \lVert F(|x| \geq t^\kappa) F(|D| \leq t^{-\lambda}) e^{it\Delta} F(|x| \leq t^\alpha) \rVert_{L^2 \to L^2}\;ds := \rmI^\text{low}_2 =& o(1)
\end{split}\end{equation}
We begin with $\rmI^\text{low}_2$. Observe that we can write
\begin{equation*}
    F(|x| \geq t^\kappa) F(|D| \leq t^{-\lambda}) e^{i(t-s)\Delta} F(|x| \leq s^\beta) \phi = \int K^\text{low}_2(x,y) \phi(y)\;dy
\end{equation*}
where
\begin{equation*}
    K^\text{low}_2(x,y) = \frac{F(|x| \geq t^\kappa)F(|y| \leq s^\beta)}{2\pi} \int e^{-it\xi^2 + (x-y)\xi} F(|\xi| \leq t^{-\lambda})\;d\xi
\end{equation*}
Integrating repeatedly by parts, we find that
\begin{equation*}\begin{split}
    K^\text{low}_2(x,y) =& \frac{F(|x| \geq t^\kappa)F(|y| \leq s^\beta)}{2\pi} \int e^{-it\xi^2 + (x-y)\xi} \left(\partial_\xi \frac{-i}{2t\xi + (x-y)}\right)^{K} F(|\xi| \leq t^{-\lambda})\;d\xi\\
    =& \sum_{a+b = K} C_{a,b}F(|x| \geq t^\kappa)F(|y| \leq s^\beta) \int e^{-i\editadd{t}\xi^2 + (x-y)\xi} \frac{-i(s-t)^a}{\left(2t\xi + (x-y)\right)^{K+a}} t^{b\lambda} F^{(b)}(|\xi| \leq t^{-\lambda})\;d\xi
\end{split}\end{equation*}
Now, since $\kappa > \max(\beta, 1-\lambda)$, we have that
\begin{equation*}
    |\editadd{-}2t\xi + (x-y)| \sim |x|
\end{equation*}
so
\begin{equation*}\begin{split}
    |K^\text{low}_2(x,y)| \lesssim_K& \sum_{a+b = K} F(|x| \geq t^\kappa)F(|y| \leq s^\beta) \frac{t^at^{(b-1)\lambda}}{|x|^{(K+a)}}
\end{split}
\end{equation*}
Integrating in space and time, we see that
\begin{equation*}
    \int_T^t\lVert K^\text{low}_2(x,y) \rVert_{L^2_{x,y}}\;ds \lesssim_K \sum_{a,b =K} \frac{T^{a+1 + (b-1)\lambda + \beta/2}}{T^{\kappa(K+a-1/2)}}
\end{equation*}
Since $\kappa > 1/2 > \lambda$, this expression vanishes as $T \to \infty$, establishing the second bound in~\eqref{eqn:rem-low-op-bds}.  To prove the first bound, we observe that $\rmI^\text{low}_1$ is associated with the integral kernel
\begin{equation*}\begin{split}
    K^\text{low}_1(x,y) =& \frac{F(|x| \geq t^\kappa) F(|y| \leq t^\alpha)}{2\pi} \int e^{-i(t-s)\xi^2 + i(x-y)\xi} F(|\xi| \leq t^{-\lambda})\;d\xi
\end{split}\end{equation*}
Since $\kappa > \alpha$, a straightforward adaptation of the arguments used to control $K^\text{low}_2$ shows that
\begin{equation*}
    \int_T^t\lVert K^\text{low}_1(x,y) \rVert_{L^2_{x,y}}\;ds = o(1)
\end{equation*}
as required.

\paragraph{\textit{Bounds for incoming waves}}
Now, we turn to the bounds for $P^\text{in}\urem$.  Note that $P^\text{in}$ projects to incoming waves.  Heuristically, we should not expect to find incoming waves after the solution has evolved forward in time from data localized to near the origin, so we will again use the Duhamel representation from the past~\eqref{eqn:u-rem-0-Duhamel}.  Expanding out the relevant operators, we find that
\begin{equation*}\begin{split}
    F(|x| \geq t^\kappa) P^\text{in}\urem =& -i \int_T^t F(x \geq t^\kappa) F(-D \geq t^{-\lambda}) e^{it\Delta}  F(|x| \geq t^\alpha) e^{-is\Delta} F(|x| \leq s^\beta) G(u(s)) \;ds\\
    &+ \{\text{symmetric and better terms}\}\\
    &= -i \int_T^t F(x \geq t^\kappa) F(-D \geq t^{-\lambda}) e^{i(t-s)\Delta} F(|x| \leq s^\beta) G(u(s)) \;ds\\
    &+i \int_T^t F(x \geq t^\kappa) e^{it\Delta}F(-D \geq t^{-\lambda})  F(|x| \leq t^\alpha) e^{-is\Delta} F(|x| \leq s^\beta) G(u(s)) \;ds\\
    &+ \{\text{symmetric and better terms}\}\\
    &:= \rmI^\text{in}_1 + \rmI^\text{in}_2 + \{\text{symmetric and better terms}\}\\
\end{split}\end{equation*}
The estimate for $\rmI^\text{in}_1$ is similar to those used for $\rmI^\text{low}_1$ and $\rmI^\text{low}_2$.  Introducing the kernel $K^\text{in}_1$ corresponding to the operator in $\rmI^\text{in}_1$, we see that
\begin{equation*}\begin{split}
    K^\text{in}_1(x,y) =& \frac{F(x \geq t^\kappa) F(|y| \leq s^\beta)}{2\pi} \int e^{-i(t-s)\xi^2 + i(x-y)\xi} F(-\xi \geq t^{-\lambda})\;d\xi\\
    =& \frac{F(x \geq t^\kappa) F(|y| \leq s^\beta)}{2\pi} \int e^{-i(t-s)\xi^2 + i(x-y)\xi} \left(\partial_\xi \frac{i}{2(t-s)\xi - (x-y)}\right)^K F(-\xi \geq t^{-\lambda})\;d\xi\\
    =& \sum_{a+b = K} C_{a,b} F(x \geq t^\kappa) F(|y| \leq s^\beta) \int e^{-i(t-s)\xi^2 + i(x-y)\xi} \frac{i(t-s)^a}{(2(t-s)\xi - (x-y))^{K+a}} t^{b\lambda} F^{(b)}(-\xi \geq t^{-\lambda})\;d\xi
\end{split}\end{equation*}
Observing that $|2(t-s)\xi - (x-y)| \sim \max ((t-s)|\xi|, |x|)$, we see that
\begin{equation*}\begin{split}
    \int \frac{(t-s)^a}{|2(t-s)\xi - (x-y)|^{K+a}} t^{b\lambda} |F^{(b)}(-\xi \geq t^{-\lambda})|\;d\xi  \lesssim& \frac{(t-s)^{a-1}}{|2(t-s) t^{-\lambda} - (x-y)|^{K+a-1}} t^{b\lambda}\\
    \lesssim& \frac{t^{(K-1)\lambda}}{|x|^{K}}
\end{split}\end{equation*}
Thus, we find that
\begin{equation*}\begin{split}
    \lVert K^\text{in}_1(x,y) \rVert_{L^2_{x,y}} \lesssim& \sum_{a+b = K} C_{a,b}  \lVert F(x \geq t^\kappa) F(|y| \leq s^\beta) \frac{t^{(K-1) \lambda}}{|x|^K}\rVert_{L^2_{x,y}}\\
    \lesssim&_K t^{(K-1)(\lambda - \kappa)} t^{\frac{\beta-\kappa}{2}}
\end{split}\end{equation*}
Since $\kappa > \lambda$, this can be made to decay arbitrarily fast in $t$, and $\lVert \rmI^\text{in}_1\rVert_{L^2} = o(1)$.  The argument that $\rmI^\text{in}_2$ goes to zero is analogous, since the kernel for $F(x \geq t^\kappa) e^{it\Delta}F(-D \geq t^{-\lambda})  F(|x| \leq t^\alpha)$ is
\begin{equation*}
    K^\text{in}_{2} = \frac{F(x \geq t^\kappa)F(|y| \leq t^\alpha)}{2\pi} \int e^{-it\xi^2 + \xi(x-y)} F(-\xi \geq t^{-\lambda})\;d\xi
\end{equation*}
and for $\kappa > \alpha$, the phase derivative $-2t\xi + (x-y) \gtrsim \editadd{-2t\xi + x > 0}$.  Integrating by parts as above, we find that $\lVert F(x \geq t^\kappa) e^{it\Delta}F(-D \geq t^{-\lambda})  F(|x| \leq t^\alpha) \rVert_{L^2\to L^2}$ decays rapidly in $t$, so $\rmI^\text{in}_2 \to 0$.

\paragraph{\textit{Bounds for outgoing waves}}
It only remains to control the outgoing part of $\urem$.  Using the Duhamel representation from the future~\eqref{eqn:u-rem-infty-Duhamel-processed}, we write
\begin{equation*}\begin{split}
    F(|x| \geq t^\kappa) P^\text{out}\urem =& i\int_t^\infty F(x \geq t^\kappa) F(D \geq t^{-\lambda}) e^{it\Delta} F(|x| \geq s^\alpha) e^{-is\Delta}  F(|x| \leq s^\beta)G(u)\;ds\\
    &+ \{\text{symmetric and better terms}\}\\
    =& i\int_t^\infty F(x \geq t^\kappa) F(D \geq t^{-\lambda}) e^{i(t-s)\Delta} F(|x| \leq s^\beta)G(u)\;ds\\
    &-i\int_t^\infty F(x \geq t^\kappa) F(D \geq t^{-\lambda}) e^{it\Delta} F(|x| \leq s^\alpha) e^{-is\Delta}  F(|x| \leq s^\beta)G(u)\;ds\\
    &+ \{\text{symmetric and better terms}\}\\
    =:& \rmI^\text{out}_1 + \rmI^\text{out}_2 + \{\text{symmetric and better terms}\}
\end{split}
\end{equation*}
For $\rmI^\text{out}_1$, we observe that the kernel for $F(x \geq t^\kappa) F(D \geq t^{-\lambda}) e^{i(t-s)\Delta} F(|x| \leq s^\beta)$ is
\begin{equation*}\begin{split}
    K^\text{out}_1(x,y) =& \frac{F(x \geq t^\kappa) F(|y| \leq s^\beta)}{2\pi} \int e^{-i(t-s)\xi^2 + i\xi(x-y)} F(\xi \geq t^{-\lambda})\;d\xi\\
    =& \sum_{a+b = K} C_{a,b} \frac{F(x \geq t^\kappa) F(|y| \leq s^\beta)}{2\pi} \int e^{-i(t-s)\xi^2 + i\xi(x-y)} \frac{(t-s)^at^{b\lambda} F^{(b)}(\xi \geq t^{-\lambda})}{2(s-t)\xi + (x-y))^{K+a}}  \;d\xi
\end{split}\end{equation*}
Now, if $t \leq s \leq 2t$, then $s^\beta \ll t^\kappa$ for $t \gg 1$, so
\begin{equation*}
    |2(s-t)\xi + (x-y)| \gtrsim t^\kappa
\end{equation*}
and by the same arguments as for $\rmI^\text{in}_1$, we find that
\begin{equation*}
    \int_t^{2t} \lVert F(x \geq t^\kappa) F(D \geq t^{-\lambda}) e^{i(t-s)\Delta} F(|x| \leq s^\beta) \rVert_{L^2 \to L^2}\;ds = o(1)
\end{equation*}
On the other hand, if $s > 2t$, then $2(s-t)\xi \gtrsim st^{-\lambda} \gg s^\beta$ since $1- \lambda > \beta$, and
\begin{equation*}
    |2(s-t)\xi + (x-y)| \gtrsim st^{-\lambda}
\end{equation*}
Hence, in this regime, 
\begin{equation*}\begin{split}
    \lVert K^\text{out}_1(x,y) \rVert_{L^2_{x,y}} \lesssim_K \frac{s^{\beta/2} t^{(2K - 3/2)\lambda}}{s^{K-1/2}}
\end{split}
\end{equation*}
which is sufficient since $\lambda < 1/2$.  Turning to $\rmI^\text{out}_2$, we observe that by~\eqref{eqn:approx-comm-1},
\begin{equation*}\begin{split}
    F(D \geq t^{-\lambda}) F(|x| \leq s^\alpha) = F(D \geq t^{-\lambda}) F(|x| \leq s^\alpha) F(D \geq \frac{1}{100} t^{-\lambda}) + O_{L^2 \to L^2}(s^{-N})
\end{split}\end{equation*}
Thus, it suffices to obtain bounds on the operator $F(|x| \leq s^\alpha) F(D \geq \frac{1}{100} t^{-\lambda}) e^{-is\Delta}  F(|x| \leq s^\beta)$, which can be represented using the kernel
\begin{equation*}
    K^\text{out}_2(x,y) = \frac{F(|x| \leq s^\alpha)F(|y| \leq s^\beta)}{2\pi} \int e^{is\xi^2 + i\xi(x-y)} F(\xi \geq \frac{1}{100} t^{-\lambda})\;d\xi
\end{equation*}
Since the derivative of the phase is bounded below by a multiple of $s\xi$ under the assumptions that $\alpha + \lambda < 1, \beta + \lambda < 1$, we can integrate by parts to show that the operator decays rapidly in $s$.

\subsubsection{Improved \texorpdfstring{$\dot{H}^1$}{dot-H 1} control}
We now sketch the modifications necessary to prove~\eqref{eqn:u-rem-dot-H1-bound}.  Let us introduce a reparametrized incoming/outgoing decomposition
\begin{equation*}\begin{split}
    \tilde{P}^\text{out} =&  \bbOne_{x \geq 0} F(D \geq t^{-\nu}) + \bbOne_{x < 0} F(-D \geq t^{-\nu})\\
    \tilde{P}^\text{in} =&  \bbOne_{x \geq 0} F(-D \geq t^{-\nu}) + \bbOne_{x < 0} F(+D \geq t^{-\nu})\\
    \tilde{P}^\text{low} =& F(|D| \leq t^{-\nu})
\end{split}\end{equation*}
where
\begin{equation}
    0 < \nu < \max(1/2, 1-\alpha, \alpha, \beta, 1-\beta, \mu)
\end{equation}
Note that in general, $\nu \neq \lambda$.  Evidently,
\begin{equation*}
    \lVert \tilde{P}_\text{low} \partial_x \urem \rVert_{L^2} \to 0
\end{equation*}
so we only need bounds on $\tilde{P}_\text{in} \partial_x \urem$ and $\tilde{P}_\text{out} \partial_x \urem$.  We first observe that the estimates for the incoming and outgoing waves continue to hold if we replace $\kappa$ with $\mu$ and $\lambda$ with $\nu$, since we only used the hypotheses $\lambda > 1- \kappa$ and $\kappa > \alpha$ to deal with low frequencies.  Moreover, after expanding $\tilde{P}_\text{in} \partial_x \urem$ and $\tilde{P}_\text{out} \partial_x \urem$, we obtain two types of terms: terms where the $F(|x| \leq t^\alpha)$ multiplier is replaced by $F'(|x| \leq t^\alpha) t^{-\alpha}$, and terms where $G(u(s))$ is replaced by $\partial_x G(u(s))$.  Both types of terms have the same or better decay and localization properties (note that $\partial_x G(u(s))$ and $G(u(s))$ satisfy the same $L^2$ bounds), so we only need to make minor modifications to the previous argument to prove~\eqref{eqn:u-rem-dot-H1-bound}.

\section{Proof of Theorem~\ref{thm:wave-op-and-bdd-state}}
We now optimize the parameters to show that the previous results imply~\Cref{thm:wave-op-and-bdd-state}.  Fix $\kappa > 1/2$ and $\mu > \max(\frac{1}{3}, \frac{1}{\sigma})$.  Observe that we can choose $\alpha < \min(\alpha_0,\kappa)$, $\delta < \min(1/2, \alpha)$ and take 
\begin{equation*}
    \min\left(\frac{1}{3}, \frac{1}{\sigma}\right) < \beta < \min(\kappa, \mu)
\end{equation*}
and the hypotheses for~\eqref{eqn:alpha-delta-restrs} and~\eqref{eqn:beta-def} are satisfied.  Thus, by~\Cref{thm:wave-op-exists}, we can write
\begin{equation*}
    u(t) = e^{it\Delta} u_+ + \uloc(t)
\end{equation*}
where $u_+ = \Omega_\text{free} u_0$ and
\begin{equation}\label{eqn:u-loc-def}
    \uloc(t) = u(t) - e^{it\Delta} u_+ = (I- \Jfree) u(t) + e^{it\Delta}(e^{-it\Delta} \Jfree u(t) -  u_+)
\end{equation}
By the definition of $u_+$, the second term in~\eqref{eqn:u-loc-def} goes to $0$ in $H^1$, while the first term obeys the bounds~\eqref{eqn:u-nonfree-bd-L2} and~\eqref{eqn:u-nonfree-bd-dot-H1} by~\Cref{thm:loc-state-spread}, which implies~\Cref{thm:wave-op-and-bdd-state}.

\section*{Declarations}

\paragraph{\textbf{Funding:}} A. Soffer is supported in part by NSF-DMS Grant number 2205931.

\paragraph{\textbf{Competing interests:}} The authors have no relevant financial or non-financial interests to disclose.

\bibliography{sources}
\bibliographystyle{plain}

\end{document}